\documentclass[11pt]{article}
\usepackage{latexsym}
\usepackage{amsmath}
\usepackage{amsthm,color}
\usepackage{amssymb}
\usepackage{mathrsfs}
\usepackage{cases}

\usepackage{lscape}
\newtheorem{theorem}{\indent Theorem}[section]

\newtheorem{definition}[theorem]{\indent Definition}
\newtheorem{lemma}[theorem]{\indent Lemma}
\newtheorem{remark}[theorem]{\indent Remark}

\begin{document}
\renewcommand{\baselinestretch}{1.3}


\begin{center}
    {\large \bf On desingularization of steady vortex for the lake equations}
\vspace{0.5cm}\\{\sc Daomin Cao*}\\
{\small Academy of Mathematics and Systems Science, CAS,Beijing, 100190, China}\\
{\small and University of Chinese Academy of Sciences, Beijing 100049,  P.R. China}\\
{\sc Weicheng Zhan}\\
{\small Academy of  Mathematics and Systems Science, CAS, Beijing 100190}\\
{\small and University of Chinese Academy of Sciences, Beijing 100049,  P.R. China}\\
and\\
{\sc Changjun Zou}\\
{\small Academy of  Mathematics and Systems Science, CAS, Beijing 100190}\\
{\small and University of Chinese Academy of Sciences, Beijing 100049,  P.R. China}\\
\end{center}


\renewcommand{\theequation}{\arabic{section}.\arabic{equation}}
\numberwithin{equation}{section}


\begin{abstract}
In this paper, we constructed a family of steady vortex solutions for the lake equations with a general vorticity function, which constitutes a desingularization of a singular vortex. The precise localization of the asymptotic singular vortex is shown to be the deepest position of the lake. We also study global nonlinear stability for these solutions. Some qualitative
and asymptotic properties are also established.

\textbf{Keywords:} Steady vortex solutions; Elliptic equations; Variational method; Desingularization;
Maximizer
\end{abstract}

\vspace{-1cm}

\footnote[0]{ \hspace*{-7.4mm}
$^{*}$ Corresponding author.\\
E-mails: dmcao@amt.ac.cn; zhanweicheng16@mails.ucas.ac.cn; zouchangjun17@mails.ucas.ac.cn}

\section{Introduction}
In this paper, we are concerned with a two-dimensional geophysical model which describes the evolution of the vertically averaged horizontal component of the three-dimensional velocity of an incompressible Euler flow.
The lake equations in a planar domain $D$ with prescribed initial and boundary conditions are

\begin{numcases}
{}
\label{1-1} \text{div}(b\,\mathbf{v})=0\,\ \, \ \ \ \ \ \ \  \ \, &\text{on}\  $\mathbb{R_+}\times D$,  \\
\label{1-2} \partial_t\mathbf{v}+(\mathbf{v}\cdot\nabla)\mathbf{v}=-\nabla P \ \ \ \  &\text{on}\  $\mathbb{R_+}\times D$,\\
\label{1-3}  (b\mathbf{v})\cdot \mathbf{n}=0\ \ \ \ \ \ \ \ \ \ \ \ \ \ \, &\text{on}\  $\mathbb{R_+}\times \partial D$,\\
\label{1-in} \mathbf{v}(0,x)=\mathbf{v}_0(x)\ \ &\text{on}\ D.
\end{numcases}

Here $\mathbf{v}=\mathbf{v}(t,x)=(v_1,v_2)$ is the velocity field, $P=P(t,x)$ is the scalar pressure, $b=b(x)$ is a positive depth function, and $\mathbf{n}$ is the unit outward normal to $\partial D$. The two-dimensional Euler equations and the three-dimensional axisymmetric Euler equations are two particular cases of this lake model. For more background on this model, we refer to \cite{Cam, De3}.

Let $\omega=\text{curl}\mathbf{v}:=\partial_1v_2-\partial_2v_1$ be the corresponding vorticity. Applying curl on the momentum equation \eqref{1-2}, we get the following vorticity equation
\begin{equation}\label{tran}
  \partial_t\omega+\text{div}(\omega\mathbf{v})=0.
\end{equation}
Thanks to the condition $\text{div}(b\mathbf{v})=0$, equation \eqref{tran} is equivalent to
\begin{equation}\label{tran2}
  \partial_t\omega+b\mathbf{v}\cdot \nabla\left(\frac{\omega}{b}\right)=0.
\end{equation}
Let us introduce the potential vorticity $\zeta:=b^{-1}\omega$. Then equation \eqref{tran2} becomes
\begin{equation}\label{vor}
  \partial_t \zeta+\mathbf{v} \cdot \nabla \zeta=0,
\end{equation}
which is actually a nonlinear transport equation. By virtue of equation \eqref{1-1}, there exists a stream function $\psi$ such that
\begin{equation}\label{1-4}
  \mathbf{v}=\left(\frac{\partial_2\psi}{b},-\frac{\partial_1\psi}{b}\right).
\end{equation}
Now we turn to consider the steady lake equations. Substituting  \eqref{1-4} into equation \eqref{vor} we obtain
\begin{equation}\label{vor2}
 \nabla^\perp \psi \cdot\nabla \zeta=0,
\end{equation}
where $x^\perp:=(x_2,-x_1)$ denotes clockwise rotation through ${\pi}/{2}$. Equation \eqref{vor2} suggests that $\psi$ and $\zeta$ are functionally dependent.
Indeed if $\zeta=f(\psi)$ for some vorticity function $f: \mathbb{R}\to \mathbb{R}$, then equation \eqref{vor2} automatically holds. Let $F$ be a primitive function of $f$. Once we find the stream function $\psi$, the velocity of the flow is given by \eqref{1-4} and the pressure is given by $P=-F(\psi)-\frac{1}{2}|\mathbf{v}|^2$. We are thus interested in studying the asymptotic behavior of solutions of
\begin{equation}\label{1-5}
  \mathcal{L}\psi:=-\frac{1}{b}\,\text{div}\frac{\nabla \psi}{b}=\frac{1}{\varepsilon^2}f(\psi).
\end{equation}
In \cite{DV}, De Valeriola and Van Schaftingen studied this problem. They constructed a family of desingularized solutions of equation \eqref{1-5} when $f(s)=s_+^p$ for some $p>1$. Their method was based on the mountain pass theorem of Ambrosetti and Rabinowitz \cite{Am}. Recently, Dekeyser also investigated desingularization of the steady lake equations via a different method \cite{De1, De2}. He proposed a variational principle (called the vorticity method) by following the lines of thought of Arnol'd \cite{Ar1,Ar2,Ar3} and Benjamin \cite{Be}. Using the vorticity method, he obtained a family of desingularized vortex pairs by maximization of  the kinetic energy over a class of rearrangements of sign changing functions. Compared to the work of De Valeriola and Van Schaftingen \cite{DV}, the vorticity distribution function is now prescribed. However, the vorticity function $f$ in \cite{De1, De2} arises as an infinite-dimensional Lagrange multiplier corresponding the vorticity rearrangement class and hence left undetermined. For the general nonlinearity $f$, it seems that there are not any results. Our purpose in this paper is to study desingularization of the steady lake equations with general vorticity function $f$. The class of admissible functions of $f$ considered here includes all $s_+^p$ for $p>0$. More precisely, we make the following assumptions on $f$:
\begin{itemize}
\item[(H1)] $f$ is continuous on $\mathbb{R}$, $f(s)=0$ for $s\leq 0$, and $f$ is strictly increasing in $(0,+\infty)$.
\item[(H2)] There exists $\vartheta_0\in(0,1)$ such that
\[\int_0^sf(r)dr\leq \vartheta_0f(s)s,\,\,\forall\,s\geq0.\]
\item[(H3)] For all $\tau>0,$
\[\lim_{s\to+\infty}f(s)e^{-\tau s}=0.\]
\end{itemize}
To state our main results, we need to make some preparations first. Lebesgue measure on $\mathbb{R}^2$ is denoted by $\textit{m}$, and is to be understood as the measure defining any $L^p$ space and $W^{1,p}$ space, except when stated otherwise; $\nu$ denotes the measure on $\mathbb{R}^2$ having density $b$ with respect to $\textit{m}$ and $|\cdot|$ denotes $\nu$-measure; $B_\delta(y)$ denotes an open ball in $\mathbb{R}^2$ of center $y$ and radius $\delta>0$.
We define the inverse of $\mathcal{L}$ as follows. One can check that the operator $\mathcal{K}$ is well-defined; see, e.g., \cite{De1, DV}.
\begin{definition}
The Hilbert space $H(D)$ is the completion of $C_0^\infty(D)$ with the scalar product
\begin{equation*}
  \langle u,v\rangle_H=\int_D\frac{1}{b^2}\nabla u\cdot\nabla v d\nu.
\end{equation*}
We define the inverse $\mathcal{K}$ for $\mathcal{L}$ in the weak solution sense,
\begin{equation}\label{2-1}
  \langle \mathcal{K}u,v\rangle_H=\int_D uv d\nu \ \  \text{for all}\  v\in H(D), \ \ when \ u \in  L^p(D,d\nu)\ \text{for some}\ p>1.
\end{equation}
\end{definition}

 Let $D\subset \mathbb{R}^2$ be a bounded simply-connected Lipschitz domain and positive depth function $b\in C^\alpha(\bar{D})\cap C^1_{\text{loc}}(D)$ for some $\alpha \in (0,1)$. As in
\cite{De2}, we introduce the following definition.
\begin{definition}
 A lake $(D,b)$ is said to be continuous if the operator $\mathcal{K}$ admits the following integral kernel representation:
 \begin{equation*}
   \mathcal{K}\zeta(x)=b(x)\int_D G(x,y)\zeta(y)d\nu(y)+\int_D R(x,y)\zeta(y)d\nu(y)\  \ \ \forall\, \zeta\in L^p(D, d\nu),\, p>1,
 \end{equation*}
 where $G$ is the Green's function for $-\Delta$ with Dirichlet boundary conditions, and $R: D\times D\to \mathbb{R}$ is a bounded and continuous correction function.
\end{definition}
In fact, more general concept of continuous lake is defined in \cite{De2}. For simplicity, we will consider only the above case here. We remark that this class covers two situations encountered in the literature: first, the case of $b\in W^{1,\infty}(D)\cap C^1_{\text{loc}}(D)$ with the additional condition that $\inf_D b>0$; second, the case of a degenerated depth function that vanishes as a polynomial of some regularized distance at the boundary. Mixed conditions are also possible. For more examples, we refer the reader to \cite{De1}. Finally, let $\mathcal{S}:=\{x\in \bar{D}\mid b(x)=\sup_D b\}$.

Having made these preparations, we now state our first main result.
\begin{theorem}\label{thm1}
Let $(D,b)$ be a continuous lake. Let $f$ be a function satisfying (H1)--(H3) and $\kappa>0$. Suppose $\mathcal{S}\cap D \not=\varnothing$, then for all sufficiently small $\varepsilon>0$, there exists a family of solutions $(\psi^\varepsilon,\zeta^\varepsilon)$ with the following properties:
\begin{itemize}
\item[(i)]For any $p>1$, $\psi^\varepsilon\in W^{2,p}_{\text{loc}}(D)$ and satisfies
\begin{equation*}
   \mathcal{L}\psi^\varepsilon =\zeta^\varepsilon\ \ \text{a.e.} \ \text{in} \ D. \\
\end{equation*}
\item[(ii)] $(\psi^\varepsilon,\zeta^\varepsilon)$ is of the form
\begin{equation*}
      \psi^\varepsilon=\mathcal{K}\zeta^\varepsilon-\mu^\varepsilon,\ \  \zeta^\varepsilon=\frac{1}{\varepsilon^2}f(\psi^\varepsilon),\ \ \int_D\zeta^\varepsilon d\nu=\kappa,
\end{equation*}
for some $\mu^\varepsilon>0$ depending on $\varepsilon$. Furthermore, as $\varepsilon \to 0^+$, one has
\begin{equation*}
   \mu^\varepsilon= \frac{\kappa \sup_D b}{2\pi}\ln \frac{1}{\varepsilon}+O(1).
 \end{equation*}

 \item[(iii)] There exists a constant $\eta>0$ not depending on $\varepsilon$ such that $\text{dist}\left(\text{supp}(\zeta^\varepsilon),\partial D\right)>\eta$. Moreover, there exists some constant $L_0>1$ independent of $\varepsilon$ such that
 \begin{equation*}
   {diam}\left(\text{supp}(\zeta^\varepsilon)\right)\le L_0\varepsilon.
 \end{equation*}
 As $\varepsilon$ goes to zero, $\text{supp}(\zeta^\varepsilon)$ will shrink to $\mathcal{S}$. That is, for every $\delta>0$, there holds
\[\mbox{\text{supp}}(\zeta^\varepsilon)\subset \mathcal{S}_\delta:=\{x\in D \mid \text{dist}(x,\mathcal{S})<\delta\}\]
provided $\varepsilon>0$ is sufficiently small.
\item[(iv)]Let the center of vorticity be
\begin{equation*}
X^\varepsilon=\frac{\int_D x\zeta^\varepsilon(x)dm(x)}{\int_D \zeta^\varepsilon(x)dm(x)},
\end{equation*}
and define the rescaled version of $\zeta^\varepsilon$ to be
\begin{equation*}
\xi^\varepsilon(x)={\varepsilon^2}\zeta^\varepsilon(X^\varepsilon+\varepsilon x)
\end{equation*}
Then every accumulation points of $\xi^\varepsilon(x)$ as $\varepsilon \to 0^+$, in the weak topology of $L^2$, are radially nonincreasing functions.
\end{itemize}
\end{theorem}

The condition $\mathcal{S}\cap D \not=\varnothing$ means that there exist some deepest points inside the lake. In such a case, the steady vortices $\{\zeta^\varepsilon: \varepsilon>0\}$ will be away from the boundary and concentrated in one of the deepest position inside the lake. If $b$ is constant, Theorem \ref{thm1} does not help to locate the vortex; indeed a refined estimate for the Euler equation will locate them at minimum of the Robin function of $D$ \cite{Cao2}. Moreover, one may expect to obtain a more accurate picture of the asymptotic behavior of the vortex by further studying the second order asymptotic properties of the energy functional, see \cite{De2} for some discussions. When $f(s)=s_+^p$ for some $p>1$, our results are similar to the results of De Valeriola and Van Schaftingen \cite{DV}. However, we would like to point out that the circulation of the flow we constructed is constant while theirs is not. It seems difficult to construct such solutions in their way.

Our second main result is concerned with the other case when $\mathcal{S}\subset \partial D$. That is, the deepest position is only reached on the boundary $\partial D$.

\begin{theorem}\label{thm2}
Let $(D,b)$ be a continuous lake. Let $f$ be a function satisfying (H1)--(H3) and $\kappa>0$. Suppose $\mathcal{S}\subset \partial D$, then for all sufficiently small $\varepsilon>0$, there exists a family of solutions $(\psi^\varepsilon,\zeta^\varepsilon)$ with the following properties:
\begin{itemize}
\item[(i)]For any $p>1$, $\psi^\varepsilon\in W^{2,p}_{\text{loc}}(D)$ and satisfies
\begin{equation*}
   \mathcal{L}\psi^\varepsilon =\zeta^\varepsilon\ \ \text{a.e.} \ \text{in} \ D. \\
\end{equation*}
\item[(ii)] $(\psi^\varepsilon,\zeta^\varepsilon)$ is of the form
\begin{equation*}
      \psi^\varepsilon=\mathcal{K}\zeta^\varepsilon-\mu^\varepsilon,\ \  \zeta^\varepsilon=\frac{1}{\varepsilon^2}f(\psi^\varepsilon),\ \ \int_D\zeta^\varepsilon d\nu=\kappa,
\end{equation*}
for some $\mu^\varepsilon>0$ depending on $\varepsilon$. Furthermore, as $\varepsilon \to 0^+$, one has
\begin{equation*}
   \mu^\varepsilon= \frac{\kappa \sup_D b}{2\pi}\ln \frac{1}{\varepsilon}+O\left(\ln \ln \frac{1}{\varepsilon}\right).
 \end{equation*}

 \item[(iii)] There holds
 \begin{equation*}
   \lim_{\varepsilon \to 0^+}\frac{\ln \text{diam} \left(\text{supp}(\zeta^\varepsilon)\right)}{\ln \varepsilon}=1.
 \end{equation*}
  As $\varepsilon$ goes to zero, $\text{supp}(\zeta^\varepsilon)$ will shrink to $\mathcal{S}$. That is, for every $\delta>0$, there holds
\[\mbox{\text{supp}}(\zeta^\varepsilon)\subset \mathcal{S}_\delta\]
provided $\varepsilon>0$ is sufficiently small.

\item[(iv)]Let the center of vorticity be
\begin{equation*}
X^\varepsilon=\frac{\int_D x\zeta^\varepsilon(x)dm(x)}{\int_D \zeta^\varepsilon(x)dm(x)},
\end{equation*}
and define the rescaled version of $\zeta^\varepsilon$ to be
\begin{equation*}
\xi^\varepsilon(x)={\varepsilon^2}\zeta^\varepsilon(X^\varepsilon+\varepsilon x)
\end{equation*}
Then every accumulation points of $\xi^\varepsilon(x)$ as $\varepsilon \to 0^+$, in the weak topology of $L^2$, are radially nonincreasing functions.
\end{itemize}
\end{theorem}

Compared to Theorem \ref{thm1}, in this case some fine estimates were lost. This is mainly because the interaction of the vortex with the boundary lost some energy of order $O\left(\ln \ln {\varepsilon}\right)$. We do not study this phenomenon in detail in the present work. If $b(x)=x_1$ and $f(s)=s_+$, this model also occurs in the plasma problem, see, e.g., \cite{Ber,CF1, Te}. It describes the equilibrium of a plasma confined in a toroidal cavity (a ``Tokamak machine''). In \cite{CF1}, Caffarelli and Friedman studied asymptotic behavior of this system. They constructed a family of plasmas which were shown to converge to the part of the boundary of the domain. One can see that our results are similar to theirs.

To explain our strategy let us introduce some notations. For a given function $f$ satisfying $(H1)$, let $f^{-1}$ be defined as the inverse function of $f$ in $(0,+\infty)$ and $f^{-1}\equiv0$ in $(-\infty,f(0^+)]$. Set $F(s):=\int_{0}^{s}f(r)dr$ and  the conjugate function to $F$ is then defined by $F_*(s)=\int_0^sf^{-1}(r)dr$. Our strategy for the proofs of Theorems \ref{thm1} and \ref{thm2} is as follows. We construct solutions by the variational method. As in \cite{Cao2}\,(see also \cite{Ber}), we introduce two functionals as follows
\begin{equation*}
  E(\zeta):=\frac{1}{2}\int_D \zeta(x)\mathcal{K}\zeta(x)d\nu(x),\ \ \mathcal{F}_\varepsilon(\zeta):=\frac{1}{\varepsilon^2}\int_D F_*(\varepsilon^2\zeta(x))d\nu(x).
\end{equation*}
Then we will show existence of a maximizer of the variational integral
\begin{equation*}
  \mathcal{E}(\zeta)=E(\zeta)-\mathcal{F}_\varepsilon(\zeta)
\end{equation*}
over the following admissible class
\begin{equation*}
\mathcal{A}_{\varepsilon}:=\{\omega\in L^\infty(D)~|~ 0\le \omega \le \frac{\Lambda}{\varepsilon^2}~ \mbox{ a.e. in }D, \int_{D}\omega(x)d\nu(x)=\kappa \},
\end{equation*}
where $\kappa>0$, $\varepsilon>0$, and $\Lambda>1$ is a sufficiently large real number such that $\mathcal{A}_{\varepsilon}$ is not empty. We will show that each maximizer of the energy functional $\mathcal{E}$ over the admissible class $\mathcal{A}_{\varepsilon}$ will yield a desired solution when $\Lambda$ is large enough and $\varepsilon$ is sufficiently small. Since we require that $\mathcal{A}_{\varepsilon}$ is contained in some bounded subset of $ L^\infty(D)$, an absolute maximum for $\mathcal{E}$ over $\mathcal{A}_{\varepsilon}$ can be easily found. However, this restriction may affect the equation that the critical point satisfies. Fortunately, we can prove that the maximizer does not touch this ``boundary'' of $\mathcal{A}_{\varepsilon}$ if $\Lambda$ is chosen to be large enough at first. In the study of asymptotic behavior of $\zeta^\varepsilon$ when $\varepsilon\to 0^+$, our key idea is to expand the energy as precisely as possible. To maximize the energy, these solutions must be concentrated. The proof is an adaptation of techniques of Turkington \cite{Tur83}\,(see also Dekeyser \cite{De2}).

Having constructed the above steady solutions, we are interested in their nonlinear stability. Under some assumptions, we will prove that these steady solutions are stable for the vorticity dynamics \eqref{tran}. The stability is of Liapunov type: it is global in time in the $L^p$-norm on the vorticities. Recall that we can reconstruct the velocity by \eqref{1-4} via the vorticity-stream system $(\zeta, \psi)$. Hence we are led to the following vorticity formulation

\begin{equation}\label{time}
\begin{cases}
  &\partial_t\omega+\text{div}(\omega\mathbf{v})=0\ \ \text{in}\ D, \\
  &\ \mathbf{v}=b^{-1}\nabla^\perp \mathcal{K}\left(\frac{\omega}{b}\right).
\end{cases}
\end{equation}

We interpret this equation in the distributional sense. Recall that $\omega=b\,\zeta$.
\begin{definition}
  Given an initial data $\zeta_0\in L^\infty(D)$, we say that $\zeta \in L^\infty([0,\infty)\times D, \mathbb{R})$ is a weak solution if for every test function $\phi \in C^\infty_{\text{c}}([0,+\infty)\times D)$, one has
  \begin{equation}\label{time2}
    \begin{split}
        & \int_{D}\zeta_0(x) \phi(0,x)d\nu(x)+\int_{0}^{+\infty}\int_D\zeta(t,x)\left(\partial_t\phi(t,x)+\mathbf{v}\cdot \phi(t,x)\right)d\nu(x) dt=0, \\
         &\ \  \mathbf{v}(t,\cdot)=b^{-1}\nabla^\perp \mathcal{K}\zeta\ \ \text{for\, a.e.}\ t\in [0,+\infty).
    \end{split}
  \end{equation}
\end{definition}
A continuous lake $(D,b)$ is said to be smooth if $\partial D \in C^\infty$ and $b\in C^2(\bar{D},(0,+\infty))$. When $(D,b)$ is smooth, weak solutions of the Cauchy problem exist globally and these solutions are unique, see \cite{De2} and the references therein. For $\zeta \in L^1(D,\nu)$, we use $\mathcal{R}(\zeta)$ to denote all rearrangements of $\zeta$ with respect to measure $\nu$. We have the following stability criterion.
\begin{theorem}\label{thm3}
  Let $(D,b)$ be a smooth lake. Let $\zeta_0\in L^\infty(D)$, and suppose $\zeta_0$ is a strict local maximiser of kinetic energy $E$ relative to $\mathcal{R}(\zeta_0)$ in $L^p(D)$ for some $p\in(1,+\infty)$. Then $\zeta_0$ is a steady weak solution of equation \eqref{time2}. Moreover, $\zeta$ is stable in the following sense:

  For each $\epsilon>0$, there exists $\delta>0$ such that, if ${\zeta}\in L^\infty([0,\infty)\times D, \mathbb{R})$ is a weak solution of equation \eqref{time2}, and $\|{\zeta}(0,\cdot)-\zeta_0\|_{L^p(D)}<\delta$, then $\|{\zeta}(t,\cdot)-\zeta_0\|_{L^p(D)}<\epsilon$ for all $t>0$.
\end{theorem}

This result is in fact a counterpart of Burton's criterion \cite{Bur2} in the lake equations. Its proof is based on conservation of kinetic energy and transport of vorticity, and the idea can be traced back to Arnol'd \cite{Ar1,Ar2,Ar3}. With this criterion in hand, we can prove that the solutions constructed above are stable steady solutions under some assumptions. Indeed, one can show that $\mathcal{R}(\zeta^\varepsilon)\subset \mathcal{A}_{\varepsilon}$. Notice that $\mathcal{F}_\varepsilon$ is constant on $\mathcal{R}(\zeta^\varepsilon)$. It follows that $\zeta^\varepsilon$ is actually a maximiser of kinetic energy $E$ relative to $\mathcal{R}(\zeta^\varepsilon)$. Hence we have

\begin{theorem}\label{thm4}
  Let $(D,b)$ be a smooth lake. Let $\zeta^\varepsilon$ be the solution obtained in Theorem \ref{thm1} or Theorem \ref{thm2}. Suppose $\zeta^\varepsilon$ is a strict local maximiser of kinetic energy relative to $\mathcal{R}(\zeta^\varepsilon)$ in $L^p(D)$ for some $p\in(1,+\infty)$. Then $\zeta^\varepsilon$ is a steady weak solution of equation \eqref{time2}. Moreover, $\zeta^\varepsilon$ is stable in the following sense:

  For each $\epsilon>0$, there exists $\delta>0$ such that, if $\tilde{\zeta}\in L^\infty([0,\infty)\times D, \mathbb{R})$ is a solution of equation \eqref{time2}, and $\|\tilde{\zeta}(0,\cdot)-\zeta^\varepsilon\|_{L^p(D)}<\delta$, then $\|\tilde{\zeta}(t,\cdot)-\zeta^\varepsilon\|_{L^p(D)}<\epsilon$ for all $t>0$.
\end{theorem}

We are not able to prove that $\zeta^\varepsilon$ is a strict local maximizer. However, we would like to mention that it can be reduced to the local uniqueness for an related elliptic problem. We refer the reader to Cao et al \cite{Cao1} for a similar situation in vortex patch problem.

This paper is organized as follows. In Section 2, we give the proof of Theorem \ref{thm1}. The proof of Theorem \ref{thm2} is presented in Section 3, which is similar to the previous case but the computations turn out to be more involved due to the interaction of the vortex with the boundary. In Section 4, we give the proofs of Theorems \ref{thm3} and \ref{thm4}.

\section{Proof of Theorem \ref{thm1}}
In this section, we prove Theorem \ref{thm1}. To do this, we split the proof into several lemmas.

 Notice that assumption (H2) implies $\lim_{s\to +\infty}f(s)=+\infty$ (see \cite{Ni}). It is not difficult to verifies that (H2) is in fact equivalent to
\begin{itemize}
 \item[ (H2)$'$] There exists $\vartheta_1\in(0,1)$ such that
\[F_*(s) \ge \vartheta_1 s\, f^{-1}(s),\,\,\forall\,s\geq0.\]
\end{itemize}

\subsection{Variational problem}
 Let $\kappa>0$ be fixed and $\varepsilon>0$ be a parameter. Define
\begin{equation*}
\mathcal{A}_{\varepsilon,\Lambda}:=\{\zeta\in L^\infty(D)~|~ 0\le \zeta \le \frac{\Lambda}{\varepsilon^2}~ \mbox{ a.e. in }D, \int_{D}\zeta(x)d\nu (x)=\kappa \},
\end{equation*}
where $\Lambda>\max\{1,\kappa \varepsilon^2/|D|\}$ is a positive number. Note that $\mathcal{A}_{\varepsilon,\Lambda}$ is not empty.
Consider the maximization problem of the following functional over $\mathcal{A}_{\varepsilon,\Lambda}$
$$\mathcal{E}(\zeta)=\frac{1}{2}\int_D \zeta(x)\mathcal{K}\zeta(x)d\nu(x)-\frac{1}{\varepsilon^2}\int_D F_*(\varepsilon^2\zeta(x))d\nu(x),\,\,\zeta\in \mathcal{A}_{\varepsilon,\Lambda}.$$
Recall that
\begin{equation*}
  E(\zeta):=\frac{1}{2}\int_D \zeta(x)\mathcal{K}\zeta(x)d\nu(x),\ \ \mathcal{F}_\varepsilon(\zeta):=\frac{1}{\varepsilon^2}\int_D F_*(\varepsilon^2\zeta(x))d\nu(x).
\end{equation*}
One see that $\mathcal{F}_\varepsilon$ is a convex functional over $\mathcal{A}_{\varepsilon,\Lambda}$. Denote
\begin{equation*}
  K(x,y)=b(x)G(x,y)+R(x,y).
\end{equation*}
An absolute maximum for $\mathcal{E}$ over $\mathcal{A}_{\varepsilon,\Lambda}$  can be easily found.

\begin{lemma}\label{lem1}
	$\mathcal{E}$ is bounded from above and attains its maximum value over $\mathcal{A}_{\varepsilon,\Lambda}$.
\end{lemma}

\begin{proof}
Since $G\in L^1(D\times D)$ and $R\in L^1(D\times D)$, we have
\begin{equation*}
  E(\zeta)\leq \frac{\Lambda^2(\max_D b)^3}{2\varepsilon^4}\|G\|_{L^1(D\times D)}+\frac{\Lambda^2(\max_D b)^2}{2\varepsilon^4}\|R\|_{L^1(D\times D)},\ \ \forall\,\zeta\in \mathcal{A}_{\varepsilon,\Lambda}.
\end{equation*}
On the other hand, we have
\begin{equation*}
  |\mathcal{F}_\varepsilon(\zeta)|\leq \frac{1}{\varepsilon^2}F_*(\Lambda)|D|,\,\,\ \ \forall\,\zeta\in \mathcal{A}_{\varepsilon,\Lambda}.
\end{equation*}
Therefore $\mathcal{E}$ is bounded from above over $\mathcal{A}_{\varepsilon,\Lambda}$. Let $\{\zeta_{j}\}\subset \mathcal{A}_{\varepsilon,\Lambda}$ such that as $j\to +\infty$
$$\mathcal{E}(\zeta_{j}) \to \sup_{\zeta\in \mathcal{A}_{\varepsilon,\Lambda}}\mathcal{E}({\zeta}).$$
Since $\mathcal{A}_{\varepsilon,\Lambda}$ is a sequentially compact subset of $L^\infty(D)$ in the weak-star topology, we may assume that, up to a subsequence, $\zeta_j\to\bar{\zeta}$ weakly star in $L^\infty(D)$ as $j\to\infty$ for some $\bar{\zeta}\in \mathcal{A}_{\varepsilon,\Lambda}$. Next we show that $\bar{\zeta}$ is in fact a maximizer of $\mathcal{E}$ over $\mathcal{A}_{\varepsilon,\Lambda}$.  To this end, it suffices to prove
	\[\mathcal{E}(\bar{\zeta})\geq\limsup_{j\to\infty}\mathcal{E}(\zeta_j).\]
	Since $K(\cdot,\cdot)\in L^{1}(D\times D)$, there holds
	\begin{equation}\label{Evar}
	\lim_{j\to\infty}E(\zeta_j)=E(\bar{\zeta}).
	\end{equation}
	On the other hand, since $\mathcal{F}_\varepsilon$ is a convex functional, one has
	\begin{equation}\label{Fvar}
	\liminf_{j\to +\infty} \mathcal{F}_\varepsilon(\zeta_j)\ge \mathcal{F}_\varepsilon(\bar\zeta).
	\end{equation}
	Combining \eqref{Evar} and \eqref{Fvar} we get the desired result.
\end{proof}

The next lemma gives the profile of the maximizer of $\mathcal{E}$. We shall use $\chi_{_A}$ to denote the characteristic function of a given set $A\subset\mathbb{R}^2$.
\begin{lemma}\label{lem2}
	Let $\zeta^{\varepsilon,\Lambda}$ be a maximizer of $\mathcal{E}$ over $\mathcal{A}_{\varepsilon,\Lambda}$. Then there exists a number $\mu^{\varepsilon,\Lambda}\in \mathbb{R}$ such that
	\begin{equation}\label{2-3}
	\zeta^{\varepsilon,\Lambda}=\frac{1}{\varepsilon^2}f(\psi^{\varepsilon,\Lambda}){\chi}_{_{\{x\in D\mid0<\psi^{\varepsilon,\Lambda}(x)<f^{-1}(\Lambda)\}}}+\frac{\Lambda}{\varepsilon^2}\chi_{_{\{x\in D\mid\psi^{\varepsilon,\Lambda}(x) \geq f^{-1}(\Lambda)\}}} \ \ \mbox{ a.e. in }  D,
	\end{equation}
	where
	\begin{equation}\label{2-4}
	\psi^{\varepsilon,\Lambda}:=\mathcal{K}\zeta^{\varepsilon,\Lambda}-\mu^{\varepsilon,\Lambda}.
	\end{equation}
	Moreover, $\mu^{\varepsilon,\Lambda}$ has the following lower bound
	\begin{equation}\label{2-5}
	\mu^{\varepsilon,\Lambda} \ge -f^{-1}(\Lambda)-\kappa \|R\|_{L^{\infty}(D\times D)}.
	\end{equation}
\end{lemma}

\begin{proof}
	Consider a family of variations of $\zeta^{\varepsilon,\Lambda}$ as follows
	\begin{equation*}
	\zeta_{(s)}=\zeta^{\varepsilon,\Lambda}+s({\zeta}-\zeta^{\varepsilon,\Lambda}),\ \ \ s\in[0,1],
	\end{equation*}
	where ${\zeta}$ is an arbitrary element of $\mathcal{A}_{\varepsilon,\Lambda}$. Since $\zeta^{\varepsilon,\Lambda}$ is a maximizer, we have
	\begin{equation*}
	0  \ge \frac{d\mathcal E(\zeta_{(s)})}{ds}\bigg|_{s=0^+}
	=\int_{D}({\zeta}-\zeta^{\varepsilon,\Lambda})\left(\mathcal{K}\zeta^{\varepsilon,\Lambda}-f^{-1}(\varepsilon^2\zeta^{\varepsilon,\Lambda})\right)d\nu(x),
	\end{equation*}
	that is,
	\begin{equation*}
	\int_{D}\zeta^{\varepsilon,\Lambda}\left(\mathcal{K}\zeta^{\varepsilon,\Lambda}-f^{-1}(\varepsilon^2\zeta^{\varepsilon,\Lambda})\right)d\nu(x)\ge \int_{D}{\zeta}\left(\mathcal{K}\zeta^{\varepsilon,\Lambda}-f^{-1}(\varepsilon^2\zeta^{\varepsilon,\Lambda})\right)d\nu(x),
	\end{equation*}
	for all ${\zeta}\in \mathcal{A}_{\varepsilon,\Lambda}.$
	By an adaptation of the bathtub principle (see Lieb and Loss \cite{LL}, \S 1.14), we obtain
	\begin{equation}\label{2-6}
\begin{cases}
  ~\mathcal{K}\zeta^{\varepsilon,\Lambda}-\mu^{\varepsilon,\Lambda} \ge f^{-1}(\varepsilon^2\zeta^{\varepsilon,\Lambda})\,\, \,\, &\text{whenever}\,\, ~ \zeta^{\varepsilon,\Lambda}=\frac{\Lambda}{\varepsilon^2}, \\
 ~ \mathcal{K}\zeta^{\varepsilon,\Lambda}-\mu^{\varepsilon,\Lambda} = f^{-1}(\varepsilon^2\zeta^{\varepsilon,\Lambda})\,\,\,\, &\text{whenever}\,\,\, 0<\zeta^{\varepsilon,\Lambda}<\frac{\Lambda}{\varepsilon^2}, \\
 ~ \mathcal{K}\zeta^{\varepsilon,\Lambda}-\mu^{\varepsilon,\Lambda}  \le f^{-1}(\varepsilon^2\zeta^{\varepsilon,\Lambda})\,\,\,\,&\text{whenever}\,\,\,  \zeta^{\varepsilon,\Lambda}=0,
\end{cases}
	\end{equation}
	where $\mu^{\varepsilon,\Lambda}$ is a real number determined by
	$$\mu^{\varepsilon,\Lambda}=\inf\left\{s\in\mathbb R\mid|\{x\in D\mid\mathcal{K}\zeta^{\varepsilon,\Lambda}-f^{-1}(\varepsilon^2\zeta^{\varepsilon,\Lambda})>s\}|\le \frac{\kappa\varepsilon^2}{\Lambda}\right\}.$$
Now the desired form $\eqref{2-3}$ follows immediately.
	
	Next we prove \eqref{2-5}. Suppose not, then for any $x\in D$ we have
 \begin{equation*}
 \begin{split}
    \psi^{\varepsilon,\Lambda}(x)&=\mathcal{K}\zeta^{\varepsilon,\Lambda}(x)-\mu^{\varepsilon,\Lambda}\\
         & =b(x)\int_D G(x,y)\zeta^{\varepsilon,\Lambda}(y)d\nu(y)+\int_D R(x,y)\zeta^{\varepsilon,\Lambda}(y)d\nu(y)-\mu^{\varepsilon,\Lambda}\\
         & \ge \int_D R(x,y)\zeta^{\varepsilon,\Lambda}(y)d\nu(y)+\kappa\|R\|_{L^{\infty}(D\times D)}+f^{-1}(\Lambda)\\
         & \ge f^{-1}(\Lambda),
 \end{split}
 \end{equation*}
 which implies $\zeta^{\varepsilon,\Lambda}=\Lambda\varepsilon^{-2}\chi_{_D}$. Recalling $\int_D\zeta^{\varepsilon,\Lambda}d\nu=\kappa$, we derive $\Lambda=\kappa\varepsilon^2/|D|$, which leads to a contradiction. The proof is thus completed.
\end{proof}	

\subsection{Limiting behavior}
In the following we analyze the limiting behavior of $\zeta^{\varepsilon,\Lambda}$ when $\varepsilon\to 0^+$. As mentioned before, the key idea is to estimate the order of energy as optimally as possible. To begin with, we give a lower bound of $\mathcal{E}(\zeta^{\varepsilon,\Lambda})$.

\begin{lemma}\label{lem3}
For all sufficiently small $\varepsilon>0$, we have
\begin{equation*}
  \mathcal{E}(\zeta^{\varepsilon,\Lambda})\ge \frac{\kappa^2\sup_D b}{4\pi}\ln{\frac{1}{\varepsilon}}-C,
\end{equation*}
where  $C>0$ does not depend on $\varepsilon$ and $\Lambda$.
\end{lemma}

\begin{proof}
The idea is to choose some suitable admissible functions. Fix $\bar{x}\in \mathcal{S}\cap D$ and $0<\delta<\text{dist}(\bar{x},\partial D)$.
Let $b_0=\inf_{B_\delta(\bar{x})}b$ and set
\begin{equation*}
  \tilde{\zeta}^{\varepsilon,\Lambda}=\frac{b_0}{\varepsilon^2b} \chi_{_{B_{\varepsilon \sqrt{\kappa/\pi b_0}}(\bar{x})}}.
\end{equation*}
It is clear that $\tilde{\zeta}^{\varepsilon,\Lambda} \in\mathcal{A}_{\varepsilon,\Lambda}$ for all sufficiently small $\varepsilon>0$. Since $\zeta^{\varepsilon,\Lambda}$ is a maximizer, we have $\mathcal{E}(\zeta^{\varepsilon,\Lambda})\ge \mathcal{E}(\tilde{\zeta}^{\varepsilon,\Lambda})$. A simple calculation yields to
	\begin{equation*}
	\begin{split}
	\mathcal{E}(\tilde{\zeta}^{\varepsilon,\Lambda})&=\frac{1}{2}\int_D\int_D K(x,y)\tilde{\zeta}^{\varepsilon,\Lambda}(x)\tilde{\zeta}^{\varepsilon,\Lambda}(y)d\nu(x)d\nu(y)-\frac{1}{\varepsilon^2}\int_D F_*\left(\varepsilon^2\tilde{\zeta}^{\varepsilon,\Lambda}(x)\right)d\nu(x) \\	&\ge\frac{1}{2}\int_D\int_D\frac{b(x)}{2\pi}\ln\frac{1}{|x-y|}\tilde{\zeta}^{\varepsilon,\Lambda}(x)\tilde{\zeta}^{\varepsilon,\Lambda}(y)d\nu(x)d\nu(y)-C\\
	&\ge \frac{\kappa^2}{4\pi}\left(b(\bar{x})-\left(\varepsilon \sqrt{\kappa/\pi b_0}\right)^{\alpha}\|b\|_{C^{\alpha}(\bar{D})}\right)\ln{\frac{1}{\varepsilon}}-C\\
	&\ge \frac{\kappa^2b(\bar{x})}{4\pi}\ln{\frac{1}{\varepsilon}}-C,\\
	\end{split}
	\end{equation*}
	where the positive number $C$ neither depends on $\varepsilon$ nor $\Lambda$. Thus the proof is completed.
\end{proof}
We now turn to estimate the Lagrange multiplier $\mu^{\varepsilon,\Lambda}$.

\begin{lemma}\label{lem4}
There holds
\begin{equation*}
  \mu^{\varepsilon,\Lambda}\ge \frac{\kappa \sup_D b}{2\pi}\ln{\frac{1}{\varepsilon}}-|1-2\vartheta_1|f^{-1}(\Lambda)-C,
\end{equation*}
where $\vartheta_1$ is the positive number in (H2)$'$, and the constant $C>0$ does not depend on $\varepsilon$ and $\Lambda$.
\end{lemma}

\begin{proof}
	Recalling \eqref{2-3} and the assumption (H2)$'$ on $f$, we have
	\begin{equation}\label{2-7}
	\begin{split}
	2\mathcal E(\zeta^{\varepsilon,\Lambda}) &=  \int_D\zeta^{\varepsilon,\Lambda}\mathcal{K}\zeta^{\varepsilon,\Lambda}d\nu-\frac{2}{\varepsilon^2}\int_DF_*(\varepsilon^2\zeta^{\varepsilon,\Lambda})d\nu \\
	& \le\int_D \zeta^{\varepsilon,\Lambda} \psi^{\varepsilon,\Lambda}d\nu-2\vartheta_1\int_D \zeta^{\varepsilon,\Lambda} f^{-1}(\varepsilon^2\zeta^{\varepsilon,\Lambda})d\nu+\kappa\mu^{\varepsilon,\Lambda} \\
	&=\int_{\{0<\zeta<{\Lambda}{\varepsilon^{-2}}\}}\zeta^{\varepsilon,\Lambda} f^{-1}(\varepsilon^2\zeta^{\varepsilon,\Lambda})d\nu+\int_{\{\zeta={\Lambda}{\varepsilon^{-2}}\}}\zeta^{\varepsilon,\Lambda} \psi^{\varepsilon,\Lambda}d\nu\\
	&\ \ \  -2\vartheta_1\int_D \zeta^{\varepsilon,\Lambda} f^{-1}(\varepsilon^2\zeta^{\varepsilon,\Lambda})d\nu+\kappa\mu^{\varepsilon,\Lambda} \\
	&=\int_{D}\zeta^{\varepsilon,\Lambda} f^{-1}(\varepsilon^2\zeta^{\varepsilon,\Lambda})d\nu-\int_{\{\zeta={\Lambda}{\varepsilon^{-2}}\}}\zeta^{\varepsilon,\Lambda}f^{-1}(\Lambda)d\nu\\
	&\ \ \ +\int_{\{\zeta=\Lambda \varepsilon^{-2}\}}\zeta^{\varepsilon,\Lambda} \psi^{\varepsilon,\Lambda}d\nu-2\vartheta_1\int_D \zeta^{\varepsilon,\Lambda} f^{-1}(\varepsilon^2\zeta^{\varepsilon,\Lambda})d\nu+\kappa\mu^{\varepsilon,\Lambda}\\
	& \le |1-2\vartheta_1|\kappa  f^{-1}(\Lambda)+\int_D\zeta^{\varepsilon,\Lambda}\left(\psi^{\varepsilon,\Lambda}-f^{-1}(\Lambda)\right)_+d\nu +\kappa\mu^{\varepsilon,\Lambda}\\
&\le |1-2\vartheta_1|\kappa  f^{-1}(\Lambda)+\int_D\zeta^{\varepsilon,\Lambda}\left(\psi^{\varepsilon,\Lambda}-f^{-1}(\Lambda)-2\kappa \|R\|_{L^{\infty}(D\times D)}\right)_+d\nu\\
    &\ \ \ +2\kappa^2 \|R\|_{L^{\infty}(D\times D)} +\kappa\mu^{\varepsilon,\Lambda}.
	\end{split}
	\end{equation}
	Set $U^{\varepsilon,\Lambda}:=\left(\psi^{\varepsilon,\Lambda}-f^{-1}(\Lambda)-2\kappa \|R\|_{L^{\infty}(D\times D)}\right)_+$. Since $\mu^{\varepsilon,\Lambda}\geq-f^{-1}(\Lambda)-\kappa \|R\|_{L^{\infty}(D\times D)}$, we have $U^{\varepsilon,\Lambda}\in H(D)$.
	So by integration by parts we have
	\begin{equation}\label{eeeee}
	\int_D\frac{|\nabla U^{\varepsilon,\Lambda}|^2}{b^2}d\nu= \int_D \zeta^{\varepsilon,\Lambda}U^{\varepsilon,\Lambda}d\nu.
	\end{equation}
	Then by H\"older's inequality and Sobolev's inequality
	\begin{equation}\label{fffff}
	\begin{split}
	\int_D \zeta^{\varepsilon,\Lambda}U^{\varepsilon,\Lambda}d\nu
	& \le \frac{\Lambda}{\varepsilon^2}|\{x\in D\mid\zeta^{\varepsilon,\Lambda}={\Lambda}{\varepsilon^{-2}}\}|^{\frac{1}{2}}\left(\int_D |U^{\varepsilon,\Lambda}|^2d\nu\right)^{\frac{1}{2}}\\
	& \le \frac{C\Lambda }{\varepsilon^2}|\{x\in D\mid\zeta^{\varepsilon,\Lambda}={\Lambda}{\varepsilon^{-2}}\}|^{\frac{1}{2}}\left(\int_D |\nabla U^{\varepsilon,\Lambda}|d\textit{m} \right)\\
	& \le \frac{C\Lambda}{\varepsilon^2}|\{x\in D\mid\zeta^{\varepsilon,\Lambda}={\Lambda}{\varepsilon^{-2}}\}|\left(\int_D {\frac{|\nabla U^{\varepsilon,\Lambda}|^2}{b^2}d\nu}\right)^{\frac{1}{2}}.\\
	& \le C\kappa\left(\int_D {\frac{|\nabla U^{\varepsilon,\Lambda}|^2}{b^2}d\nu}\right)^{\frac{1}{2}}.\\
	\end{split}
	\end{equation}
	Here the constant $C>0$  does not depend on $\varepsilon$ and $\Lambda$. From \eqref{eeeee} and \eqref{fffff} we conclude that $\int_D \zeta^{\varepsilon,\Lambda}U^{\varepsilon,\Lambda}d\nu$ is uniformly bounded with respect to $\varepsilon$ and $\Lambda$, which together with \eqref{2-7} and Lemma \ref{lem3} leads to the desired result.
\end{proof}

To prove that vortices $\{\zeta^{\varepsilon,\Lambda}\}$ will be concentrated when $\varepsilon$ tends to zero, we need the following technical lemma.
\begin{lemma}[\cite{Cao2}]\label{lem5}
	Let $\Omega\subset D$, $0<\epsilon<1$, $A\ge0$, and let non-negative $\Gamma: D \to \mathbb{R}$ satisfy $\|\Gamma\|_{L^1(D)}=1$ and $\|\Gamma\|_{L^p(D)}\le C_1 \epsilon^{-2(1-{1}/{p})}$ for some $1<p\le +\infty$ and $C_1>0$. Suppose for any $x\in \Omega$, there holds
	\begin{equation}\label{2-12}
	(1-A)\ln\frac{1}{\epsilon}\le \int_D \ln\frac{1}{|x-y|}\Gamma(y)d\textit{m}(y)+C_2,
	\end{equation}
	where $C_2$ is a positive constant.
	Then there exists some constant $L>1$ (which may depend on $C_1, C_2$ but not on $A, \epsilon$) such that
	\begin{equation*}
	\text{diam}(\Omega)\le L\epsilon^{1-2A}.
	\end{equation*}
\end{lemma}

Using Lemma \ref{lem5}, we are able to show that the size of $\text{supp}(\zeta^{\varepsilon,\Lambda})$ is of order $\varepsilon$.
\begin{lemma}\label{lem6}
	There exists some $L_0>1$ independent of $\varepsilon$ such that
	\begin{equation*}
	\text{diam}\left(\text{supp}(\zeta^{\varepsilon,\Lambda})\right)\le L_0\varepsilon.
	\end{equation*}
\end{lemma}
\begin{proof}
	For each $ x\in \text{supp}(\zeta^{\varepsilon,\Lambda})$ there holds
	\begin{equation*}
 \frac{\kappa \sup_D b}{2\pi}\ln{\frac{1}{\varepsilon}}-|1-2\vartheta_1|f^{-1}(\Lambda)-C\le \mathcal{K}\zeta^{\varepsilon,\Lambda}(x)\le \frac{\sup_D b}{2\pi}\int_D \ln\frac{1}{|x-y|}\zeta^{\varepsilon,\Lambda}(y)d\nu(y)+C,
	\end{equation*}
which implies
\begin{equation*}
  \ln{\frac{1}{\varepsilon}}-C(\Lambda)\le \int_D \ln\frac{1}{|x-y|}\kappa^{-1}\zeta^{\varepsilon,\Lambda}(y)b(y)d\textit{m}(y).
\end{equation*}
Since $\int_D\kappa^{-1}\zeta^{\varepsilon,\Lambda}(y)b(y)d\textit{m}(y)=1$, by Lemma \ref{lem5}, we deduce that
\begin{equation*}
\text{diam}\left(\text{supp}(\zeta^{\varepsilon,\Lambda})\right)\le L_0\varepsilon,
\end{equation*}
where $L_0>1$ may depend on $\Lambda$, but not on $\varepsilon$. The proof is thus completed.
\end{proof}

 Let
\begin{equation*}
  H(x,y):=\frac{1}{2\pi}\ln \frac{\text{diam}(D)}{|x-y|}-G(x,y).
\end{equation*}
We have the following estimate for $H$, which is required in the further analysis (see \cite{De2}).
\begin{lemma}\label{hhh}
For all $x,y \in D$, we have
\begin{equation*}
\begin{split}
\frac{1}{2\pi}&\ln  \frac{diam(D)}{\max \{ |x-y|, dist(x,\partial D),dist(y,\partial D) \} }\\
&~~~~~~~~~~~~~~~~~~~~~~~~~~~~~~~~~~~~~~~~~\ge H(x,y)  \\
                    &~~~~~~~~~~~~~~~~~~~~~~~~~~~~~~~~~~~~~~~~~ \ge \frac{1}{2\pi}\ln\frac{diam(D)}{|x-y|+2\max \{ dist(x,\partial D),dist(y,\partial D) \} }.
\end{split}
\end{equation*}
\end{lemma}

\begin{lemma}\label{lem7}
  There exists some constant $\eta>0$ not depending on $\varepsilon$ such that for every $\Lambda>\max\{1,\kappa \varepsilon^2/|D|\}$ and all sufficiently small $\varepsilon>0$, we have $\text{dist}\left(supp(\zeta^{\varepsilon,\Lambda}),\partial D\right)>\eta$.
\end{lemma}
\begin{proof}
  We argue by contradiction. Suppose that there exists a sequence $\{\varepsilon_j\}_{j=1}^{\infty}$ such that $\varepsilon_j\to 0^+$ and $\text{supp}(\zeta^{\varepsilon_j,\Lambda})\subset D_j:=\bar{D}\cap B_{{1}/{j}}(x_0)$ for some $x_0\in \partial D$  as $j\to +\infty$. Using Lemma \ref{lem6}, it is not hard to obtain
  \begin{equation*}
      \mathcal{E}(\zeta^{\varepsilon_j,\Lambda}) \le \frac{\sup_{D_j} b}{4\pi}\ln\frac{1}{\varepsilon_j}-\inf_{D_j} b\int_D\int_D H(x,y)\zeta^{\varepsilon_j, \Lambda}(x)\zeta^{\varepsilon_j, \Lambda}(y)d\nu(x)d\nu(y)+C(\Lambda),
  \end{equation*}
  where $C(\Lambda)>0$ depends on $\Lambda$, but not on $\varepsilon_j$. Combining this and Lemma \ref{lem3}, we derive that $x_0$ must belong to $\mathcal{S}\cap \partial D$. Moreover, for all sufficiently large $j$, one has
  \begin{equation}\label{ad1}
    \int_D\int_D H(x,y)\zeta^{\varepsilon_j, \Lambda}(x)\zeta^{\varepsilon_j, \Lambda}(y)d\nu(x)d\nu(y)\le C(\Lambda).
  \end{equation}
  But, by Lemma \ref{hhh}, we have
  \begin{equation*}
    \int_D\int_D H(x,y)\zeta^{\varepsilon_j, \Lambda}(x)\zeta^{\varepsilon_j, \Lambda}(y)d\nu(x)d\nu(y) \to +\infty, \ \ \text{as}\ j \to +\infty.
  \end{equation*}
  which together with \eqref{ad1} clearly leads to a contradiction.
\end{proof}

The following lemma shows that  $\psi^{\varepsilon,\Lambda}$ has a prior upper bound with respect to $\Lambda$.

\begin{lemma}\label{lem8}
One has
\begin{equation*}
  \psi^{\varepsilon,\Lambda} \le |1-2\vartheta_1|  f^{-1}(\Lambda)+\frac{\kappa  \sup_D b}{4\pi} \ln \Lambda +C+o_\varepsilon(1),
\end{equation*}
where the constant $C>0$ does not depend on $\varepsilon$ and $\Lambda$, and $o_\varepsilon(1)\to 0$ as $\varepsilon \to 0^+$.
\end{lemma}
\begin{proof}
	For any $x\in \text{supp}(\zeta^{\varepsilon,\Lambda})$, we have
	\begin{equation*}
	\begin{split}
	\psi^{\varepsilon,\Lambda}(x) &  \le \frac{1}{2\pi}\int_D b(x)\ln\frac{1}{|x-y|}\zeta{^{\varepsilon,\Lambda}}(y)d\nu(y)-\mu^{_{\varepsilon,\Lambda}}+C \\
	& \le \frac{\sup_D b}{2\pi}\left(b(x)+(L_0\varepsilon)^\alpha \|b\|_{C^\alpha(\bar{D})} \right)\int_D \ln \frac{1}{|x-y|}\zeta^{\varepsilon,\Lambda}(y) d\textit{m}(y)-\mu^{_{\varepsilon,\Lambda}}+C \\
    & \le \frac{\sup_D b}{2\pi}\left(b(x)+(L_0\varepsilon)^\alpha \|b\|_{C^\alpha(\bar{D})} \right)\left(\ln\frac{1}{\varepsilon}+\frac{\ln \Lambda}{2} \right)\int_D\zeta^{\varepsilon,\Lambda}(y) d\textit{m}(y)-\mu^{_{\varepsilon,\Lambda}}+C \\
    & \le  \frac{\sup_D b}{2\pi} \left(\ln\frac{1}{\varepsilon}+\frac{\ln \Lambda}{2} \right)\int_D\zeta^{\varepsilon,\Lambda} (y) d\nu(y)-\mu^{_{\varepsilon,\Lambda}}+C+o_\varepsilon(1)\\
    & \le \frac{\kappa \sup_D b}{2\pi} \left(\ln\frac{1}{\varepsilon}+\frac{\ln \Lambda}{2} \right)-\mu^{_{\varepsilon,\Lambda}}+C+o_\varepsilon(1).
	\end{split}
	\end{equation*}
	By Lemma \ref{lem4} we have
	\begin{equation*}
  \psi^{\varepsilon,\Lambda} \le |1-2\vartheta_1|  f^{-1}(\Lambda)+\frac{\kappa  \sup_D b}{4\pi} \ln \Lambda +C+o_\varepsilon(1).
	\end{equation*}
    The proof is completed.
\end{proof}

As a consequence of Lemma \ref{lem8}, we can eliminate the patch part in \eqref{2-3}.
\begin{lemma}\label{lem9}
	If $\Lambda$ is sufficiently large(not depending on $\varepsilon$), then for all sufficiently small $\varepsilon>0$ we have
	\begin{equation}\label{patchvanish}
	|\{x\in D\mid\zeta^{\varepsilon,\Lambda}(x)={\Lambda}{\varepsilon^{-2}}\}|=0.
	\end{equation}
	As a consequence, $\zeta^{\varepsilon,\Lambda}$ has the form
	\begin{equation*}
	\zeta^{\varepsilon,\Lambda}=\frac{1}{\varepsilon^2}f(\psi^{\varepsilon,\Lambda}).
	\end{equation*}
\end{lemma}

\begin{proof}
	Notice that
	\begin{equation}\label{2-9}
	\psi^{\varepsilon,\Lambda}\ge f^{-1}(\Lambda)\ \ \text{on}\ \  \{x\in D\mid\zeta^{\varepsilon,\Lambda}(x)={\Lambda}{\varepsilon^{-2}}\}.
	\end{equation}
	Combining \eqref{2-9} and Lemma \ref{lem8}, we conclude that there exists some $C$ not depending on $\varepsilon$ and $\Lambda$ such that
	\begin{equation}\label{2-10}
	(1-|1-2\vartheta_1|) f^{-1}(\Lambda)\le \frac{\kappa\sup_D b}{4\pi}\ln \Lambda+C+o_\varepsilon(1)   \ \ \text{on}\ \  \{x\in D\mid\omega^{\varepsilon,\Lambda}(x)={\Lambda}{\varepsilon^{-2}}\}.
	\end{equation}
	Note that since $\vartheta_1\in(0,1),$ there holds $1-|1-2\vartheta_1|\in(0,1)$. Recall the assumption (H3) on $f$, that is,
	\begin{equation*}
	\lim_{s\to+\infty}f(s)e^{-\tau s}=0,\ \ \forall\, \tau>0,
	\end{equation*}
	which implies for each $\tau>0$
	\begin{equation}\label{2-11}
	\lim_{s\to+\infty}\left(\tau f^{-1}(s)-\ln s\right)=+\infty.
	\end{equation}
	Combining \eqref{2-10} and \eqref{2-11}, we deduce that if $\Lambda$ is chosen to be large enough, then for all sufficiently small $\varepsilon>0$, we have
	\[|\{x\in D\mid\zeta^{\varepsilon,\Lambda}(x)={\Lambda}{\varepsilon^{-2}}\}|=0.\]
	The proof is completed.
\end{proof}

In the rest of this section, we fix the parameter $\Lambda$ such that Lemma \ref{lem9} holds. To simplify notations, we shall abbreviate $(\mathcal{A}_{\varepsilon,\Lambda},\zeta^{\varepsilon,\Lambda},\psi^{\varepsilon,\Lambda}, \mu^{\varepsilon,\Lambda})$ as $(\mathcal{A}_{\varepsilon},\zeta^{\varepsilon}, \psi^{\varepsilon}, \mu^{\varepsilon})$.

Now we turn to study the location of support of $\zeta^\varepsilon$ when $\varepsilon\to 0^+$. We show that $\text{supp}(\zeta^\varepsilon)$ will shrink to $\mathcal{S}$ as $\varepsilon$ goes to zero.

\begin{lemma}\label{lem10}
For any $\delta>0$, we have $supp(\zeta^\varepsilon)\subset \mathcal{S}_\delta$ if $\varepsilon>0$ is sufficiently small.
\end{lemma}

\begin{proof}
	We argue by way of contradiction. Assume that there exists a $\gamma_0>0$ and a subsequence $\{\varepsilon_j\}_{j=1}^\infty$ such that $\varepsilon_j \to 0^+$ as $j\to +\infty$ and $\text{supp}(\zeta^\varepsilon_j)\subset B_{\gamma_0}(Z)\subset D\backslash \mathcal{S}_{\gamma_0}$. Hence
	\begin{equation*}
	\begin{split}
	\mathcal{E}(\zeta^{\varepsilon_j})&\le \frac{1}{4\pi}\int_D\int_D b(x)\ln\frac{1}{|x-y|}\zeta^{\varepsilon_j}(x)\zeta^{\varepsilon_j}(y)d\nu(x)d\nu(y)+C\\
	&=\frac{\sup_{B_{\gamma_0}(Z)}b}{4\pi}\left(\int_D\int_D\ln\frac{\varepsilon_j}{|x-y|}\zeta^{\varepsilon_j}(x)\zeta^{\varepsilon_j}(y)d\nu(x)d\nu(y)
+\kappa^2\ln\frac{1}{\varepsilon_j}\right)+C\\
	&=\frac{\kappa^2\sup_{B_{\gamma_0}(Z)}b}{4\pi}\ln\frac{1}{\varepsilon_j}+C\int_D\int_D\ln\frac{\varepsilon_j}{|x-y|}\zeta^{\varepsilon_j}(x)\zeta^{\varepsilon-j}(y)d\nu(x)d\nu(y)+C.\\
	\end{split}
	\end{equation*}
Note that
	\begin{equation*}
	\begin{split}
	\int_D\int_D\ln\frac{\varepsilon_j}{|x-y|}\zeta^{\varepsilon_j}(x)\zeta^{\varepsilon_j}(y)&d\nu(x)d\nu(y)\\
    &\le\frac{\Lambda^2 (\sup_D b)^2}{\varepsilon_j^{4}}\int_{B_{L_0 \varepsilon_j}(Z)}\int_{B_{L_0 \varepsilon_j}(Z)}\ln\frac{\varepsilon_j}{|x-y|}d\textit{m}(x)d\textit{m}(y)\\
	&=\Lambda^2 (\sup_D b)^2\int_{B_{L_0}(0)}\int_{B_{L_0}(0)}\ln\frac{1}{|x-y|}d\textit{m}(x)d\textit{m}(y)\\	
	&\le C.
	\end{split}
	\end{equation*}
	Now we have
	\begin{equation}\label{2-17}
	\mathcal{E}(\zeta^{\varepsilon_j})\le\frac{\kappa^2\sup_{B_{\gamma_0}(Z)}b}{4\pi}\ln\frac{1}{\varepsilon_j}+C.
	\end{equation}
	But on the other hand, by Lemma \ref{lem3}, there holds
\begin{equation}\label{ad2}
  \mathcal{E}(\zeta^{\varepsilon_j})\ge \frac{\kappa^2\sup_D b}{4\pi}\ln{\frac{1}{\varepsilon_j}}-C.
\end{equation}
  Combining \eqref{2-17} and \eqref{ad2}, we get a contradiction and thus finish the proof.
\end{proof}

For the energy functional we have the following asymptotic expansions.
\begin{lemma}\label{lem11}
  As $\varepsilon \to 0^+$, we have
\begin{align}
\label{ad4}  \mathcal{E}(\zeta^\varepsilon) & =\frac{\kappa^2\sup_D b}{4\pi}\ln\frac{1}{\varepsilon}+O(1), \\
\label{ad5}  \mu^\varepsilon & =\frac{\kappa \sup_D b}{2\pi}\ln\frac{1}{\varepsilon}+O(1).
\end{align}
\end{lemma}
\begin{proof}
  We first prove \eqref{ad4}. Arguing as in the proof of Lemma \ref{lem10}, we can obtain
  \begin{equation*}
    \mathcal{E}(\zeta^\varepsilon)\le \frac{\kappa^2\sup_D b}{4\pi}\ln\frac{1}{\varepsilon}+O(1),
  \end{equation*}
  which together with Lemma \ref{lem3} gives \eqref{ad4}. To prove \eqref{ad5}, we note that
  \begin{equation*}
  \begin{split}
     2\mathcal{E}(\zeta^\varepsilon) & =\int_D \zeta^\varepsilon \psi^\varepsilon d\nu+\kappa \mu^\varepsilon+O(1)\\
       & =\kappa \mu^\varepsilon+O(1),
  \end{split}
  \end{equation*}
    from which \eqref{ad5} clearly follows.
\end{proof}

We now turn to study the asymptotic shape of $\zeta^\varepsilon$. Recall that we denote the center of vorticity to be
\begin{equation*}
X^\varepsilon=\frac{\int_D x\zeta^\varepsilon(x)d\textit{m}(x)}{\int_D \zeta^\varepsilon(x)d\textit{m}(x)},
\end{equation*}
and define the rescaled version of $\zeta^\varepsilon$ as follows
\begin{equation*}
  \xi^\varepsilon(x)={\varepsilon^2}\zeta^\varepsilon(X^\varepsilon+\varepsilon x), \ \ x\in D^\varepsilon:=\{x\in \mathbb{R}^2\mid X^\varepsilon+\varepsilon x \in D\}.
\end{equation*}
It is obvious that $\text{supp}(\xi^\varepsilon)\subset B_{L_0}(0)$. For convenience, we set $\xi^\varepsilon(x)=0$ if $x\in \mathbb{R}^2 \backslash D^\varepsilon$.

We denote by $g^\varepsilon$ the symmetric radially nonincreasing Lebesgue-rearrangement of $\xi^\varepsilon$ centered at the origin. The following result determines the asymptotic nature of $\zeta^\varepsilon$ in terms of its scaled version $\xi^\varepsilon$.

\begin{lemma}\label{lem12}
	Every accumulation points of $\xi^\varepsilon(x)$ as $\varepsilon \to 0^+$, in the weak topology of $L^2(\mathbb{R}^2)$, are radially nonincreasing functions.
\end{lemma}

\begin{proof}
	Up to a subsequence we may assume that $\xi^\varepsilon\to \xi^*$ and $g^\varepsilon \to g^*$ weakly in $L^2(\mathbb{R}^2)$ as $\varepsilon\to 0^+$. By Riesz's rearrangement inequality, we first have
	\begin{equation*}
	\int_{\mathbb{R}^2}\int_{\mathbb{R}^2}\ln\frac{1}{|x-y|}\xi^\varepsilon(x)\xi^\varepsilon(y)d\textit{m}(x)d\textit{m}(y)
	\le \int_{\mathbb{R}^2}\int_{\mathbb{R}^2}\ln\frac{1}{|x-y|}g^\varepsilon(x)g^\varepsilon(y)d\textit{m}(x)d\textit{m}(y).
	\end{equation*}
	Letting $\varepsilon\to 0^+$, we get
	\begin{equation}\label{2-18}
	\int_{\mathbb{R}^2}\int_{\mathbb{R}^2}\ln\frac{1}{|x-y|}\xi^*(x)\xi^*(y)d\textit{m}(x)d\textit{m}(y)
	\le \int_{\mathbb{R}^2}\int_{\mathbb{R}^2}\ln\frac{1}{|x-y|}g^*(x)g^*(y)d\textit{m}(x)d\textit{m}(y).
	\end{equation}	
	On the other hand, define
\begin{equation*}
  \hat{\zeta}^\varepsilon(x)=\varepsilon^{-2}g^\varepsilon\left(\varepsilon^{-1}(x-X^\varepsilon)\right),\ \ x\in D,
\end{equation*}
one readily verifies
\begin{equation*}
  0\le \hat{\zeta}^\varepsilon\le \frac{\Lambda}{\varepsilon^2}\ \ \text{a.e. on}\, D, \ \ \int_D \hat{\zeta}^\varepsilon d\nu=\kappa+O(\varepsilon^\alpha).
\end{equation*}
A direct calculation then yields that as $\varepsilon \to 0^+$,
	\begin{equation*}
	\mathcal{E}({\zeta^\varepsilon})= \frac{(\sup_D b)^3}{4\pi}\int_{\mathbb{R}^2}\int_{\mathbb{R}^2}\ln\frac{1}{|x-y|}\xi^\varepsilon(x)\xi^\varepsilon(y)d\textit{m}(x)d\textit{m}(y)+\frac{\kappa^2\sup_D b}{4\pi}\ln\frac{1}{\varepsilon}+A_1(\varepsilon), \\
	\end{equation*}
	and
	\begin{equation*}
	\mathcal{E}(\hat{\zeta}^\varepsilon)= \frac{(\sup_D b)^3}{4\pi}\int_{\mathbb{R}^2}\int_{\mathbb{R}^2}\ln\frac{1}{|x-y|}{g}^\varepsilon(x){g}^\varepsilon(y)d\textit{m}(x)d\textit{m}(y)+\frac{\kappa^2\sup_D b}{4\pi}\ln\frac{1}{\varepsilon}+A_2(\varepsilon),\\
	\end{equation*}
	where
	\begin{equation*}
	\lim_{\varepsilon\to 0^+}A_1(\varepsilon)=\lim_{\varepsilon\to 0^+}A_2(\varepsilon)<\infty.
	\end{equation*}
	There is an $O(\varepsilon^\alpha)$-perturbation $\tilde{\zeta}^\varepsilon$ of $\hat{\zeta}^\varepsilon$ which belongs to $\mathcal{A}_\varepsilon$. Hence
\begin{equation*}
  \mathcal{E}({\zeta^\varepsilon})\ge \mathcal{E}(\tilde{\zeta}^\varepsilon)\ge \mathcal{E}(\hat{\zeta}^\varepsilon)+o(1).
\end{equation*}
Therefore, we conclude
\begin{equation*}
  \int_{\mathbb{R}^2}\int_{\mathbb{R}^2}\ln\frac{1}{|x-y|}\xi^\varepsilon(x)\xi^\varepsilon(y)d\textit{m}(x)d\textit{m}(y)
	\ge \int_{\mathbb{R}^2}\int_{\mathbb{R}^2}\ln\frac{1}{|x-y|}g^\varepsilon(x)g^\varepsilon(y)d\textit{m}(x)d\textit{m}(y)+o(1).
\end{equation*}
Letting $\varepsilon \to 0^+$, it follows that
\begin{equation}\label{ad3}
  	\int_{\mathbb{R}^2}\int_{\mathbb{R}^2}\ln\frac{1}{|x-y|}\xi^*(x)\xi^*(y)d\textit{m}(x)d\textit{m}(y)
	\ge \int_{\mathbb{R}^2}\int_{\mathbb{R}^2}\ln\frac{1}{|x-y|}g^*(x)g^*(y)d\textit{m}(x)d\textit{m}(y).
\end{equation}
Combining \eqref{2-18} and \eqref{ad3}, we obtain
\begin{equation*}
  	\int_{\mathbb{R}^2}\int_{\mathbb{R}^2}\ln\frac{1}{|x-y|}\xi^*(x)\xi^*(y)d\textit{m}(x)d\textit{m}(y)
	= \int_{\mathbb{R}^2}\int_{\mathbb{R}^2}\ln\frac{1}{|x-y|}g^*(x)g^*(y)d\textit{m}(x)d\textit{m}(y).
\end{equation*}
By Lemma 3.2 in Burchard and Guo \cite{BG}, we know that there exists a translation $\mathcal T$ of $\mathbb{R}^2$ such that $\mathcal T\xi^*=g^*$. Note that
	\begin{equation*}
	\int_{\mathbb{R}^2}x\xi^*(x)d\textit{m}(x)=\int_{\mathbb{R}^2}x g^*(x)d\textit{m}(x)=0.
	\end{equation*}
	Thus $\xi^*=g^*$, the proof is thus completed.	
\end{proof}
\subsection{Proof of Theorem \ref{thm1}}
Now we are ready to give the proof of Theorem \ref{thm1}.
\begin{proof}[Proof of Theorem \ref{thm1}]
It follows from the above lemmas.
\end{proof}

\section{Proof of Theorem \ref{thm2}}
In this section, we give the proof of Theorem \ref{thm2}. Compare with the preceding  one, this case is a little more complicated. Although the idea of proof is similar, some of the details are different. As before, we will split the proof into several lemmas.
\subsection{Variational problem} The variational problem is the same as in subsection 2.1. Arguing as above, we first obtain the following result.
\begin{lemma}\label{add1}
 $\mathcal{E}$ attains its maximum value over $\mathcal{A}_{\varepsilon,\Lambda}$. Let $\zeta^{\varepsilon,\Lambda}$ be a maximizer, then there exists some $\mu^{\varepsilon,\Lambda}\in \mathbb{R}$ such that
	\begin{equation}\label{add2}
	\zeta^{\varepsilon,\Lambda}=\frac{1}{\varepsilon^2}f(\psi^{\varepsilon,\Lambda}){\chi}_{_{\{x\in D\mid0<\psi^{\varepsilon,\Lambda}(x)<f^{-1}(\Lambda)\}}}+\frac{\Lambda}{\varepsilon^2}\chi_{_{\{x\in D\mid\psi^{\varepsilon,\Lambda}(x) \geq f^{-1}(\Lambda)\}}} \ \ \mbox{ a.e. in }  D,
	\end{equation}
	where
	\begin{equation}\label{add3}
	\psi^{\varepsilon,\Lambda}:=\mathcal{K}\zeta^{\varepsilon,\Lambda}-\mu^{\varepsilon,\Lambda}.
	\end{equation}
	Moreover, $\mu^{\varepsilon,\Lambda}$ has the following lower bound
	\begin{equation}\label{add4}
	\mu^{\varepsilon,\Lambda} \ge -f^{-1}(\Lambda)-\kappa \|R\|_{L^{\infty}(D\times D)}.
	\end{equation}
\end{lemma}

\subsection{Limiting behavior}
 To study the limiting behavior of $\zeta^{\varepsilon,\Lambda}$, we first give a rough lower bound of $\mathcal{E}(\zeta^{\varepsilon,\Lambda})$.

\begin{lemma}\label{lem13}
	There holds
	\begin{equation}\label{3-1}
	\mathcal{E}(\zeta^{\varepsilon,\Lambda})\ge \frac{\kappa^2\sup_D b}{4\pi}\ln\frac{1}{\varepsilon}-\frac{\kappa^2\sup_D b}{4\pi\alpha}\ln\ln \frac{1}{\varepsilon}-C,
	\end{equation}
	where the positive constant $C$ is independent of $\varepsilon$ and $\Lambda$.
\end{lemma}

\begin{proof}
Fix some $\bar{x}\in \mathcal{S}\subset \partial D$. Since $D$ is a Lipschitz domain, it satisfies an interior cone condition at $\bar{x}$. Hence we can choose a family of points $\{x^\varepsilon: \varepsilon>0\}\subset D$ such that there exists a number $\theta\in(0,1)$ such that
for all sufficiently small $\varepsilon$, we have
\begin{equation*}
 \frac{\theta}{(\ln\frac{1}{\varepsilon})^{\frac{1}{\alpha}}} \le \text{dist}(x^\varepsilon,\partial D)\le\text{dist}(x^\varepsilon, \bar{x})=\frac{1}{(\ln\frac{1}{\varepsilon})^{\frac{1}{\alpha}}}.
\end{equation*}
For a small $\delta>0$, we have $b_0:=\inf_{D_\delta}b>0$ with $D_\delta:=\bar{D}\cap B_\delta(\bar{x})$. Then we define
\begin{equation*}
  \tilde{\zeta}^{\varepsilon,\Lambda}=\frac{b_0}{\varepsilon^2 b}\chi_{_{B_{\varepsilon\sqrt{\kappa/\pi b_0}}(x^\varepsilon)}}.
\end{equation*}
One can easily see that $\tilde{\zeta}^{\varepsilon,\Lambda}\in \mathcal{A}_{\varepsilon,\Lambda}$ if $\varepsilon$ is sufficiently small. Recalling Lemma \ref{hhh}, a simple calculation yields to
\begin{equation*}
	\begin{split}
	\mathcal{E}(\tilde{\zeta}^{\varepsilon,\Lambda})&=\frac{1}{2}\int_D\int_D K(x,y)\tilde{\zeta}^{\varepsilon,\Lambda}(x)\tilde{\zeta}^{\varepsilon,\Lambda}(y)d\nu(x)d\nu(y)-\frac{1}{\varepsilon^2}\int_D F_*(\varepsilon^2\tilde{\zeta}^{\varepsilon,\Lambda}(x))d\nu(x) \\
& \ge \frac{\sup_Db-2\|b\|_{C^\alpha(\bar{D})}({\ln\frac{1}{\varepsilon}})^{-1}}{4\pi}\int_D\int_D \ln\frac{1}{|x-y|}\tilde{\zeta}^{\varepsilon,\Lambda}(x)\tilde{\zeta}^{\varepsilon,\Lambda}(y)d\nu(x)d\nu(y)\\
&\ \ \ \ \ \ \ \ \ \ \ -\frac{\sup_D b}{2}\int_D\int_D H(x,y)\tilde{\zeta}^{\varepsilon,\Lambda}(x)\tilde{\zeta}^{\varepsilon,\Lambda}(y)d\nu(x)d\nu(y)-C\\
&\ge \frac{\kappa^2\sup_D b}{4\pi}\ln\frac{1}{\varepsilon}-\frac{\kappa^2\sup_D b}{4\pi\alpha}\ln\ln \frac{1}{\varepsilon}-C,
	\end{split}
	\end{equation*}
where $C>0$ does not depend on $\varepsilon$ and $\Lambda$. Since $\zeta^{\varepsilon,\Lambda}$ is a maximizer, one has $\mathcal{E}(\zeta^{\varepsilon,\Lambda})\ge \mathcal{E}(\tilde{\zeta}^{\varepsilon,\Lambda})$, the proof is thus completed.
\end{proof}

Arguing as in the proof of Lemma \ref{lem4}, we immediately get
\begin{lemma}\label{lem14}
There exists a constant $C>0$, which does not depend on $\varepsilon$ and $\Lambda$, such that
\begin{equation*}
  \mu^{\varepsilon,\Lambda}\ge \frac{\kappa\sup_D b}{4\pi}\ln\frac{1}{\varepsilon}-\frac{\kappa\sup_D b}{4\pi\alpha}\ln\ln \frac{1}{\varepsilon}-|1-2\vartheta_1|f^{-1}(\Lambda)-C,
\end{equation*}
 where $\vartheta_1$ is the positive number in (H2)$'$.
\end{lemma}
As a corollary of Lemma \ref{lem14}, for every $\gamma\in (0,1)$, we can find a constant $C>0$, which does not depend on $\varepsilon$ and $\Lambda$, such that
\begin{equation*}
   \mu^{\varepsilon,\Lambda}\ge \frac{\kappa\sup_D b}{4\pi}(1-\frac{\gamma}{2})\ln\frac{1}{\varepsilon}-|1-2\vartheta_1|f^{-1}(\Lambda)-C.
\end{equation*}

Combining this and Lemma \ref{lem5}, we obtain that the size of $\text{supp}(\zeta^{\varepsilon,\Lambda})$ is of order $\varepsilon^{\gamma}$.
\begin{lemma}\label{lem15}
	For every $\gamma\in(0,1)$, there exists a $L_0>1$, which may depend on $\Lambda$ but not on $\varepsilon$,such that
	\begin{equation}\label{ln1}
	\text{diam}\left(\text{supp}(\zeta^{\varepsilon,\Lambda})\right)\le L_0\varepsilon^\gamma.
	\end{equation}
Consequently, we have
 \begin{equation}\label{ln2}
   \lim_{\varepsilon \to 0^+}\frac{\ln \text{diam} \left(\text{supp}(\zeta^{\varepsilon,\Lambda})\right)}{\ln \varepsilon}=1.
 \end{equation}
\end{lemma}
\begin{proof}
  It remains to prove \eqref{ln2}. Indeed, by \eqref{ln1}, one has
  \begin{equation}\label{ln3}
    \liminf_{\varepsilon \to 0^+}\frac{\ln \text{diam} \left(\text{supp}(\zeta^{\varepsilon,\Lambda})\right)}{\ln \varepsilon}\ge 1.
  \end{equation}
  On the other hand, we have
  \begin{equation*}
    \kappa =\int_D \zeta^{\varepsilon,\Lambda}(x)d\nu(x)\le \frac{C\Lambda}{\varepsilon^2} \left(\text{diam} \left(\text{supp}(\zeta^{\varepsilon,\Lambda})\right)\right)^2,
  \end{equation*}
  which implies
  \begin{equation}\label{ln4}
    \limsup_{\varepsilon \to 0^+}\frac{\ln \text{diam} \left(\text{supp}(\zeta^{\varepsilon,\Lambda})\right)}{\ln \varepsilon}\le 1.
  \end{equation}
  Combining \eqref{ln3} and \eqref{ln4}, we get \eqref{ln2}.
\end{proof}
\begin{remark}
  Using Lemma \ref{lem14} and a variant of Lemma \ref{lem5}, one can obtain a much sharper estimate. More precisely, one has
  \begin{equation*}
    \text{diam}\left(\text{supp}(\zeta^{\varepsilon,\Lambda})\right)\le L_0'\,\varepsilon\left(\ln\frac{1}{\varepsilon}\right)^{2/\alpha},
  \end{equation*}
  where $L_0'>1$ may depend on $\Lambda$, but not on $\varepsilon$.
\end{remark}

Arguing as in the proof of Lemma \ref{lem10}, we know that $\text{supp}(\zeta^{\varepsilon,\Lambda})$ will shrink to $\mathcal{S}$ when $\varepsilon$ tends to zero.
\begin{lemma}\label{lem16}
For any $\delta>0$, we have $\text{supp}(\zeta^{\varepsilon,\Lambda})\subset \mathcal{S}_\delta$ if $\varepsilon>0$ is sufficiently small.
\end{lemma}
Note that $\mathcal{S}\subset \partial D$, $\text{supp}(\zeta^{\varepsilon,\Lambda})$ will approach to the boundary of $D$. We now prove that it will not approach $\partial D$ too fast. More precisely, we have
\begin{lemma}\label{lem17}
  For all sufficiently small $\varepsilon>0$, one has
  \begin{equation*}
	\text{dist}\left(\text{supp}(\zeta^{\varepsilon,\Lambda}),\partial D\right)\ge \frac{C_1}{(\ln{\frac{1}{\varepsilon}})^{\gamma_1}},
	\end{equation*}	
where $C_1, \gamma_1>0$ may depend on $\Lambda$, but not on $\varepsilon$.
\end{lemma}

\begin{proof}
Let us fix a family of points $\{x^\varepsilon\}\subset D$ with $x^\varepsilon \in \text{supp}(\zeta^{\varepsilon,\Lambda})$. By Lemma \ref{lem16}, $b(x^\varepsilon)>\sup_D b/2$ if $\varepsilon>0$ is sufficiently small. By Lemma \ref{hhh}, we have
\begin{equation*}
\begin{split}
   \mathcal{E}(\zeta^{\varepsilon,\Lambda}) & =\frac{1}{2}\int_D\int_D K(x,y)\zeta^{\varepsilon,\Lambda}(x)\zeta^{\varepsilon,\Lambda}(y)d\nu(x)d\nu(y)-\frac{1}{\varepsilon}\int_DF_*\left(\varepsilon^2\zeta^{\varepsilon,\Lambda}(x)\right)d\nu(x) \\
     & \le \frac{\sup_Db\left(b(x^\varepsilon)+O(\varepsilon^{\frac{1}{2}})\right)^2}{4\pi}\int_D\int_D \ln\frac{1}{|x-y|}\zeta^{\varepsilon,\Lambda}(x)\zeta^{\varepsilon,\Lambda}(y)d\nu(x)d\nu(y)\\
     &\ \ \ \ \ \ \ \ \ \ \ \ \ \ \ -\frac{b(x^\varepsilon)+O(\varepsilon^\frac{1}{2})}{2}\int_D\int_D H(x,y)\zeta^{\varepsilon,\Lambda}(x)\zeta^{\varepsilon,\Lambda}(y)d\nu(x)d\nu(y)+C\\
     & \le \frac{\sup_Db\left(b(x^\varepsilon)+O(\varepsilon^{\frac{1}{2}})\right)^2}{4\pi}\left(\int_D \zeta^{\varepsilon,\Lambda}(x)d\textit{m}(x)\right)^2\ln\frac{1}{\varepsilon}\\
     &\ \ \ \ \ \ \ \ \ \ \ \ \ \ \ -\frac{\kappa^2\left(b(x^\varepsilon)+O(\varepsilon^\frac{1}{2})\right)}{4\pi}\ln \frac{\text{diam} (D)}{2\,\text{dist}(\text{supp}(\zeta^{\varepsilon,\Lambda}),\partial D)+O(\varepsilon^\frac{1}{2})}+C \\
     & \le \frac{\kappa^2 \sup_D b}{4\pi}\ln\frac{1}{\varepsilon}-\frac{\kappa^2 b(x^\varepsilon)}{4\pi}\ln \frac{\text{diam} (D)}{2\,\text{dist}(\text{supp}(\zeta^{\varepsilon,\Lambda}),\partial D)+O(\varepsilon^\frac{1}{2})}+C.
\end{split}
\end{equation*}
On the other hand, by Lemma \ref{lem13}, one has
\begin{equation*}
  \mathcal{E}(\zeta^{\varepsilon,\Lambda})\ge \frac{\kappa^2\sup_D b}{4\pi}\ln\frac{1}{\varepsilon}-\frac{\kappa^2\sup_D b}{4\pi\alpha}\ln\ln \frac{1}{\varepsilon}-C.
\end{equation*}
Therefore, we deduce that
\begin{equation*}
  \frac{\kappa^2 b(x^\varepsilon)}{4\pi}\ln \frac{\text{diam} (D)}{2\,\text{dist}(\text{supp}(\zeta^{\varepsilon,\Lambda}),\partial D)+O(\varepsilon^\frac{1}{2})}\le \frac{\kappa^2\sup_D b}{4\pi\alpha}\ln\ln \frac{1}{\varepsilon}+C,
\end{equation*}
which implies
\begin{equation*}
	\text{dist}\left(\text{supp}(\zeta^{\varepsilon,\Lambda}),\partial D\right)\ge \frac{C_1}{(\ln{\frac{1}{\varepsilon}})^{\gamma_1}},
	\end{equation*}	
where $C_1, \gamma_1>0$ may depend on $\Lambda$, but not on $\varepsilon$. The proof is completed.
\end{proof}

We are now going to eliminate the patch part in \eqref{add2}. To do this, we need to establish two lemmas first. The first one is a refined version of Lemmas \ref{lem13} and \ref{lem14}. Recall that the center of vorticity is defined by
\begin{equation*}
X^\varepsilon=\frac{\int_D x\zeta^{\varepsilon,\Lambda}(x)d\textit{m}(x)}{\int_D \zeta^{\varepsilon,\Lambda}(x)d\textit{m}(x)}.
\end{equation*}

\begin{lemma}\label{lem18}
  For all sufficiently small $\varepsilon>0$, we have
  \begin{equation}\label{add5}
    \mathcal{E}(\zeta^{\varepsilon,\Lambda})\ge \frac{\kappa^2b(X^\varepsilon)}{4\pi} \ln\frac{1}{\varepsilon}-\frac{1}{2}\kappa^2 b(X^\varepsilon)H(X^\varepsilon,X^\varepsilon)-C+o_\varepsilon(1),
  \end{equation}
  consequently, there exists a constant $C>0$, which does not depend on $\varepsilon$ and $\Lambda$, such that
  \begin{equation}\label{add6}
    \mu^{\varepsilon,\Lambda}\ge \frac{\kappa b(X^\varepsilon)}{2\pi} \ln\frac{1}{\varepsilon}-\kappa b(X^\varepsilon)H(X^\varepsilon,X^\varepsilon)-|1-2\vartheta_1|f^{-1}(\Lambda)-C+o_\varepsilon(1),
  \end{equation}
    where $\vartheta_1$ is the positive number in (H2)$'$,  and $o_\varepsilon(1)\to 0^+$ as $\varepsilon \to 0^+$.
\end{lemma}

\begin{proof}
    The proof is similar to that of Lemma \ref{lem13}. By Lemma \ref{lem16}, we have $\text{supp}(\zeta^{\varepsilon,\Lambda})\subset \mathcal{S}_\delta$ for all sufficiently small $\varepsilon>0$. Let us fix  $\delta>0$ so small that $b_0:=\inf_{\mathcal{S}_\delta}b>0$. Define
    \begin{equation*}
  \tilde{\zeta}^{\varepsilon,\Lambda}=\frac{b_0}{\varepsilon^2 b}\chi_{_{B_{\varepsilon\sqrt{\kappa/\pi b_0}}(X^\varepsilon)}}.
\end{equation*}
    Thanks to Lemma \ref{lem17}, we have $\tilde{\zeta}^{\varepsilon,\Lambda}\in \mathcal{A}_{\varepsilon,\Lambda}$ if $\varepsilon$ is sufficiently small. By the interior estimate for harmonic functions, we deduce that for all $x,y\in \text{supp}(\zeta^{\varepsilon,\Lambda})$,
\begin{equation}\label{add7}
\begin{split}
   |H(x,y)-&H(X^\varepsilon,X^\varepsilon)|\\
   &\le |H(x,y)-H(X^\varepsilon,y)|+|H(X^\varepsilon,y)-H(X^\varepsilon,X^\varepsilon)|\\
     & \le \frac{C|x-X^\varepsilon|}{\text{dist}\left(\text{supp}(\zeta^{\varepsilon,\Lambda}),\partial D\right)}|\sup_D H(\cdot,y)|+\frac{C|y-X^\varepsilon|}{\text{dist}\left(\text{supp}(\zeta^{\varepsilon,\Lambda}),\partial D\right)}|\sup_D H(X^\varepsilon,\cdot)| \\
     & \le {C_1\varepsilon^{\frac{1}{2}}{\left(\ln\frac{1}{\varepsilon}\right)^{\gamma_1}} \ln\ln \frac{1}{\varepsilon}},
\end{split}
\end{equation}
where $C_1>0$ may depend on $\Lambda$ but not on $\varepsilon$, and $\gamma_1$ is the positive number in Lemma \ref{lem17}. With \eqref{add7} in hand, we can calculate as in the proof of Lemma \ref{lem13} to obtain
  \begin{equation*}
    \mathcal{E}(\zeta^{\varepsilon,\Lambda})\ge \frac{\kappa^2b(X^\varepsilon)}{4\pi} \ln\frac{1}{\varepsilon}-\frac{1}{2}\kappa^2 b(X^\varepsilon)H(X^\varepsilon,X^\varepsilon)-C_2+o_\varepsilon(1)
  \end{equation*}
for some $C_2>0$ not depending on $\varepsilon$ and $\Lambda$. Note that \eqref{add6} follows from \eqref{add5} by  arguments the same as in the proof of Lemma \ref{lem4}. The proof is thus completed.
\end{proof}

The following lemma, a counterpart of Lemma \ref{lem8}, shows that  $\psi^{\varepsilon,\Lambda}$ has a prior upper bound with respect to $\Lambda$.
\begin{lemma}\label{lem19}
One has
\begin{equation*}
  \psi^{\varepsilon,\Lambda} \le |1-2\vartheta_1|  f^{-1}(\Lambda)+\frac{\kappa  \sup_D b}{4\pi} \ln \Lambda +C+o_\varepsilon(1),
\end{equation*}
where the constant $C>0$ does not depend on $\varepsilon$ and $\Lambda$, and $o_\varepsilon(1)\to 0$ as $\varepsilon \to 0^+$.
\end{lemma}

\begin{proof}
For any $x\in \text{supp}(\zeta^{\varepsilon,\Lambda})$, we have
	\begin{equation*}
	\begin{split}
	\psi^{\varepsilon,\Lambda}(x) &  \le \frac{b(x)}{2\pi}\int_D \ln\frac{1}{|x-y|}\zeta{^{\varepsilon,\Lambda}}(y)d\nu(y)-b(x)\int_D H(x,y)\zeta{^{\varepsilon,\Lambda}}(y)d\nu(y)-\mu^{_{\varepsilon,\Lambda}}+C \\
	& \le \frac{\left(b(X^\varepsilon)+O(\varepsilon^\frac{1}{2})\right)^2}{2\pi}\int_D \ln \frac{1}{|x-y|}\zeta^{\varepsilon,\Lambda}(y) d\textit{m}(y)-\kappa b(X^\varepsilon)H(X^\varepsilon,X^\varepsilon)-\mu^{_{\varepsilon,\Lambda}}+C \\
    & \le \frac{\left(b(X^\varepsilon)+O(\varepsilon^\frac{1}{2})\right)^2}{2\pi}\left(\ln\frac{1}{\varepsilon}+\frac{\ln \Lambda}{2} \right)\int_D\zeta^{\varepsilon,\Lambda}(y) d\textit{m}(y)\\
    &~~~~~~~~~~~~~~~~~~~~~~~~~~~~~~~~~~~~-\kappa b(X^\varepsilon)H(X^\varepsilon,X^\varepsilon)-\mu^{_{\varepsilon,\Lambda}}+C \\
    & \le \frac{\kappa b(X^\varepsilon)}{2\pi} \left(\ln\frac{1}{\varepsilon}+\frac{\ln \Lambda}{2} \right)-\kappa b(X^\varepsilon)H(X^\varepsilon,X^\varepsilon)-\mu^{_{\varepsilon,\Lambda}}+C+o_\varepsilon(1).
	\end{split}
	\end{equation*}
Combining this and Lemma \ref{lem18}, we get
	\begin{equation*}
  \psi^{\varepsilon,\Lambda}(x) \le |1-2\vartheta_1|  f^{-1}(\Lambda)+\frac{\kappa  \sup_D b}{4\pi} \ln \Lambda +C+o_\varepsilon(1).
	\end{equation*}
    The proof is therefore completed.
\end{proof}
With Lemma \ref{lem19} in hand, we can now eliminate the patch part in \eqref{add2}. The proof is the same as before, we omit it here.
\begin{lemma}\label{lem20}
  If $\Lambda$ is sufficiently large(not depending on $\varepsilon$), then for all sufficiently small $\varepsilon>0$ we have
	\begin{equation}\label{patchvanish}
	|\{x\in D\mid\zeta^{\varepsilon,\Lambda}(x)={\Lambda}{\varepsilon^{-2}}\}|=0.
	\end{equation}
	As a consequence, $\zeta^{\varepsilon,\Lambda}$ has the form
	\begin{equation*}
	\zeta^{\varepsilon,\Lambda}=\frac{1}{\varepsilon^2}f(\psi^{\varepsilon,\Lambda}).
	\end{equation*}
\end{lemma}

In the rest of this section, we fix the parameter $\Lambda$ such that Lemma \ref{lem20} holds. To simplify notations, we shall abbreviate $(\mathcal{A}_{\varepsilon,\Lambda},\zeta^{\varepsilon,\Lambda},\psi^{\varepsilon,\Lambda}, \mu^{\varepsilon,\Lambda})$ as $(\mathcal{A}_{\varepsilon},\zeta^{\varepsilon}, \psi^{\varepsilon}, \mu^{\varepsilon})$. From the above proofs, it is not hard to obtain the following asymptotic expansions.
\begin{lemma}\label{lem21}
  As $\varepsilon \to 0^+$, one has
\begin{align*}
 \mathcal{E}(\zeta^\varepsilon) & =\frac{\kappa^2\sup_D b}{4\pi}\ln\frac{1}{\varepsilon}+O\left(\ln\ln\frac{1}{\varepsilon}\right), \\
  \mu^\varepsilon & =\frac{\kappa \sup_D b}{2\pi}\ln\frac{1}{\varepsilon}+O\left(\ln\ln\frac{1}{\varepsilon}\right).
\end{align*}
\end{lemma}

We now turn to determine the asymptotic shape of $\zeta^\varepsilon$. Recall that we define the rescaled version of $\zeta^\varepsilon$ as follows
\begin{equation*}
  \xi^\varepsilon(x)={\varepsilon^2}\zeta^\varepsilon(X^\varepsilon+\varepsilon x), \ \ x\in D^\varepsilon:=\{x\in \mathbb{R}^2\mid X^\varepsilon+\varepsilon x \in D\}.
\end{equation*}
For convenience, we set $\xi^\varepsilon(x)=0$ if $x\in \mathbb{R}^2 \backslash D^\varepsilon$. As before, we denote by $g^\varepsilon$ the symmetric radially nonincreasing Lebesgue-rearrangement of $\xi^\varepsilon$ centered at the origin. Define
\begin{equation*}
  \hat{\zeta}^\varepsilon(x)=\varepsilon^{-2}g^\varepsilon\left(\varepsilon^{-1}(x-X^\varepsilon)\right),\ \ x\in D.
\end{equation*}
Thanks to Lemmas \ref{lem15} and \ref{lem17}, we have
\begin{equation*}
  0\le \hat{\zeta}^\varepsilon\le \frac{\Lambda}{\varepsilon^2}\ \ \text{a.e. on}\, D, \ \ \int_D \hat{\zeta}^\varepsilon d\nu=\kappa+O(\varepsilon^{\alpha/2}).
\end{equation*}
Moreover, one has
\begin{equation*}
\begin{split}
    & \text{diam}\left(\text{supp}(\hat{\zeta}^\varepsilon)\right)\le C \varepsilon^\gamma. \ \ \forall\ \gamma\in (0,1), \\
     & \text{dist}\left(\text{supp}(\hat{\zeta}^\varepsilon),\partial D \right) \ge \frac{C_1}{(\ln{\frac{1}{\varepsilon}})^{\gamma_1}}.
\end{split}
\end{equation*}
With these results in hand, we can argue as in the proof of Lemma \ref{lem12} to obtain the following result, which determines the asymptotic nature of $\zeta^\varepsilon$ in terms of its scaled version $\xi^\varepsilon$. Since the proof is almost the same without any significant changes, we omit it.

\begin{lemma}\label{lem22}
	Every accumulation points of $\xi^\varepsilon(x)$ as $\varepsilon \to 0^+$, in the weak topology of $L^2(\mathbb{R}^2)$, are radially nonincreasing functions.
\end{lemma}
\subsection{Proof of Theorem \ref{thm2}}
\begin{proof}[Proof of Theorem \ref{thm2}]
It follows from the above lemmas.
\end{proof}

\section{Proofs of Theorems \ref{thm3} and \ref{thm4}}
In this section, we study the nonlinear stability of the above solutions. Under different assumptions, we prove that these steady solutions are stable for the vorticity dynamics \eqref{tran}. Suppose $\zeta\in L^\infty(D)$ satisfying $E(\zeta)=\sup_{\mathcal{R}(\zeta)}E$, then by Burton \cite{Bur1}, there exists some increasing function $\varphi$ such that $\zeta=\varphi(\mathcal{K}\zeta)$ almost everywhere in $D$. The following lemma provides a criterion for steady weak solutions of equation \eqref{time2}, from which it follows that such a $\zeta$ is a steady weak solution.
 For its proof, see Dekeyser \cite{De1}\,(see also Burton \cite{Bur2}).
\begin{lemma}\label{lem23}
   Let $\zeta\in L^\infty(D)$ and $\mathbf{v}=b^{-1}\nabla^\perp \mathcal{K}\zeta$. Suppose that $\zeta=\varphi(\mathcal{K}\zeta)$ a.e. in $D$ for some monotonic function $\varphi$. Then $\zeta$ is a steady weak solution of equation \eqref{time2}. That is, for all $\phi \in C^\infty_{\text{c}}(D)$, we have
  \begin{equation*}
    \int_D \zeta \mathbf{v}\cdot\nabla \phi d\nu=0.
  \end{equation*}
\end{lemma}

When (D,b) is smooth, weak solutions of the Cauchy problem exist globally and these solutions are unique. Moreover, the potential vorticity $\zeta(t,\cdot)$ at any time $t\ge 0$ is a rearrangement of the initial potential vorticity $\zeta_0$. More precisely, we have
\begin{lemma}[\cite{De2}]\label{lem24}
   Let $(D,b)$ be a smooth lake and $\zeta_0\in L^\infty(D)$. Then there exists a unique weak solution $\zeta(t,x) \in L^\infty\left([0,\infty)\times D, \mathbb{R}\right)\cap C\left([0,+\infty);L^p(D)\right)$ for all $p\in[0,+\infty)$ satisfying
   \begin{itemize}
     \item [(i)] $\zeta(t,\cdot) \in \mathcal{R}(\zeta_0)$ for all $t\ge 0$;
     \item [(ii)] the kinetic energy of the fluid is conserved, that is, $E(\zeta(t,\cdot))=E(\zeta_0)$ for all $t\ge 0$.
   \end{itemize}
\end{lemma}
We are now ready to give the proofs of Theorem \ref{thm3} and Theorem \ref{thm4}.

\subsection{Proof of Theorem \ref{thm3}}
\begin{proof}
 The proof is almost the same as in \cite{Bur2}, but we give it for the sake of completeness. Let $0<\epsilon<1$. Since $\zeta_0$ is a strict local maximiser, we can choose $0<\delta_1<\epsilon/2$ such that $\zeta_0$ maximises $E$ strictly on $\mathcal{R}(\zeta_0)\cap \mathcal{N}_\delta(\zeta_0)$, where $\mathcal{N}_\delta(\zeta_0)$ denotes an closed ball in $L^p(D,\nu)$ of center $\zeta_0$ and radius $\delta$. Since the relative weak topology of $L^p(D,\nu)$ on $\mathcal{R}(\zeta_0)$ coincides with the strong topology and $E$ is weakly sequentially continuous, we deduce that there exists a $\eta>0$ such that
\begin{equation}\label{addd1}
  E(g)<E(\zeta_0)-\eta,\ \ \forall\,g\in{\mathcal{R}(\zeta_0)\cap \partial \mathcal{N}_\delta(\zeta_0)}.
\end{equation}

Now choose $0<\delta<\delta_1$, such that if $\zeta_1,\zeta_2\in \mathcal{N}_{\|\zeta\|_{L^p(D,\nu)}+1}(0)$ and $\|\zeta_1-\zeta_2\|_{L^p(D,\nu)}<\delta$ then $|E(\zeta_1)-E(\zeta_2)|< \eta/2$. Let ${\zeta}\in L^\infty([0,\infty)\times D, \mathbb{R})$ be a weak solution of equation \eqref{time2}, and $\|\zeta(0,\cdot)-\zeta_0\|_{L^p(D,\nu)}<\delta$. Let $\mathbf{v}=b^{-1}\nabla^\perp \mathcal{K}{\zeta}$ be the corresponding velocity, and let $\tilde{\zeta}$ be the solution of $\partial_t(b\,\tilde{\zeta})+\text{div}(b\,\tilde{\zeta}\mathbf{v})=0$ with initial data $\tilde{\zeta}(0,\cdot)=\zeta_0$. Then we have
\begin{equation*}
  \|\tilde{\zeta}(t,\cdot)-\zeta(t,\cdot)\|_{L^p(D,\nu)}=\|\tilde{\zeta}(0,\cdot)-\zeta(0,\cdot)\|_{L^p(D,\nu)}<\delta.
\end{equation*}
Thus
\begin{equation}\label{addd2}
\begin{split}
     |E(\tilde{\zeta}(t,\cdot))-E(\zeta_0)| & \le |E(\tilde{\zeta}(t,\cdot))-E(\zeta(t,\cdot))|+|E(\zeta(0,\cdot))-E(\zeta_0)| \\
     &< \eta/2+\eta/2=\eta.
\end{split}
\end{equation}
Combining \eqref{addd1} and \eqref{addd2}, we derive $\tilde{\zeta}(t,\cdot)\not\in\partial \mathcal{N}_\delta(\zeta_0)$ for all $t\ge 0$, so by continuity $\tilde{\zeta}(t,\cdot)\in \mathcal{N}_\delta(\zeta_0)$ for all $t\ge 0$. Hence
\begin{equation*}
  \|\zeta(t,\cdot)-\zeta_0\|_{{L^p(D,\nu)}}\le \|\zeta(t,\cdot)-\tilde{\zeta}(t,\cdot)\|_{{L^p(D,\nu)}}+\|\tilde{\zeta}(t,\cdot)-\zeta_0\|_{{L^p(D,\nu)}}<\delta+\delta_1<\epsilon.
\end{equation*}
Therefore our proof is completed.
\end{proof}

\subsection{Proof of Theorem \ref{thm4}}
\begin{proof}
   Note that $\mathcal{R}(\zeta^\varepsilon)\subset \mathcal{A}_{\varepsilon}$ and $\mathcal{F}_\varepsilon$ is constant on $\mathcal{R}(\zeta^\varepsilon)$. Thus $\zeta^\varepsilon$ is a maximiser of kinetic energy $E$ relative to $\mathcal{R}(\zeta^\varepsilon)$. Now the desired result follows from Theorem \ref{thm3}.
\end{proof}

Acknowledgements: D. Cao was supported by NNSF of China (grant No. 11831009) and Chinese Academy of Sciences by grant QYZDJ-SSW-SYS021. W. Zhan and C. Zou were supported by NNSF of China (grant No. 11771469) and Chinese Academy of Sciences by grant QYZDJ-SSW-SYS021.

\end{document}